\documentclass{article}

\author{Dixy Msapato}
\usepackage{amsmath,enumerate,amsthm,amsfonts,amssymb,latexsym,bbm}
\usepackage[super]{nth}
\usepackage[a4paper]{geometry}
\usepackage[hidelinks]{hyperref}
\usepackage[UKenglish]{babel}
\usepackage{float}
\usepackage{dirtytalk}
\usepackage[mathscr]{euscript} 
\usepackage{tikz}
\usepackage{tikz-cd}
\usepackage{tikz-qtree}
\usepackage{mathtools}
\usepackage{xcolor}

\newcommand{\myfrac}
    [2]{\begin{array}{@{}c@{}}#1 \\[-0.75ex]#2\end{array}}

\newdimen\R
\newdimen\T
\let\amsamp=&

\newcommand\blfootnote[1]{%
  \begingroup
  \renewcommand\thefootnote{}\footnote{#1}%
  \addtocounter{footnote}{-1}%
  \endgroup
}

\theoremstyle{definition}
\newtheorem{theorem}{Theorem}[section]
\newtheorem{definition}[theorem]{Definition}
\newtheorem{proposition}[theorem]{Proposition}
\newtheorem{remark}[theorem]{Remark}
\newtheorem{corollary}[theorem]{Corollary}

\newtheorem{lemma}[theorem]{Lemma}

\def\a{\alpha}
\def\b{\beta}
\def\d{\delta}

\def\t{\tau}

\title{The Karoubi envelope and weak idempotent completion of an extriangulated category}
\date{}

\begin{document}

\maketitle
\begin{abstract}
We show that the idempotent completion and weak idempotent completion of an extriangulated category are also extriangulated. 
\end{abstract}

{\small{\textbf{Keywords}} : Extriangulated categories, Idempotent completion, Weak idempotent completion.}

{\small{\textbf{Mathematics Subject Classification (2020)}} : 18E05 (Primary) 18G15, 18G80 (Secondary)}
\thanks{}
\blfootnote{Address: School of Mathematics, University of Leeds, Leeds, LS2 9JT, United Kingdom. \\
	Contact: mmdmm@leeds.ac.uk \\
	\thanks{This research was supported by an EPSRC Doctoral Training Partnership (reference EP/R513258/1) through the University of Leeds. The author also wishes to thank their supervisor, Bethany Marsh for their continuous support and invaluable insight.}}
\section{Introduction.}
Extriangulated categories were introduced by Nakaoka and Palu in \cite{NakaokaPalu} as a simultaneous generalisation of exact categories and triangulated categories in the context of the study of cortosion pairs. The known classes of examples of extriangulated categories include exact categories and extension-closed subcategories of triangulated categories; see \cite[Example 2.13, Remark 2.18, Proposition 3.22(1)]{NakaokaPalu}. There are also examples of extriangulated categories which are not exact or triangulated; see for example, \cite[Proposition 3.30]{NakaokaPalu}, \cite[Example 4.14 and Corollary 4.12]{ZhouZhu}.

Let $\mathscr{A}$ be an additive category. A morphism $e \colon A \rightarrow A$ in $\mathscr{A}$ is said to be \textit{idempotent} if $e^{2}=e$. The category $\mathscr{A}$ is said to be \textit{idempotent complete} if every idempotent morphism in $\mathscr{A}$ admits a kernel. Every additive category $\mathscr{A}$ can be embedded fully faithfully into an idempotent complete category $\tilde{\mathscr{A}}$ called the \textit{idempotent completion} (also called the Karoubi envelope); see for example \cite[ E.g Remark 6.3]{buhler}. An additive category $\mathscr{A}$ is called \textit{weakly idempotent complete} if every retraction has a kernel, or equivalently if every section has a cokernel \cite[Definition 7.2]{buhler}. Every small additive category can also be fully faithfully embedded into a weakly idempotent complete category $\hat{\mathscr{A}}$ called the \textit{weak idempotent completion}; see for example \cite[Remark 7.8]{buhler}.

When $\mathscr{A}$ is a triangulated category, it has been shown that the idempotent completion $\tilde{\mathscr{A}}$ is also triangulated; see \cite[Theorem 1.5]{BalSch}. It has also been shown \cite[Theorem 2.16]{liu2014idempotent} that the idempotent completion of a left triangulated category is again left triangulated, and likewise for right triangulated categories. When $\mathscr{A}$ is an exact category, it has also been shown that the idempotent completion and the weak idempotent completion are exact; see \cite[Proposition 6.13, Remark 7.8]{buhler}. In this paper we will unify these results by showing that when $\mathscr{A}$ is extriangulated then the idempotent completion and the weak idempotent completion are also extriangulated. In doing so, we also add to the family of examples of extriangulated categories. 

Independent work by \cite{wwzz} has also shown that the idempotent completion of an extriangulated category is extriangulated. Although the result is the same, our work offers a different perspective. For example, the Ext$^{1}$ functor of the idempotent completion in our paper has a different description to that of \cite{wwzz}. Our description of the biadditive functor has the advantage of allowing us to easily observe that the Ext$^{1}$-groups of an idempotent completion $\tilde{\mathscr{A}}$ (and weak idempotent completion $\hat{\mathscr{A}}$) are subgroups of the Ext$^{1}$-groups of $\mathscr{A}$. In particular, the Ext$^{1}$ bifunctor on $\tilde{\mathscr{A}}$ ( and $\hat{\mathscr{A}}$) behaves like a subbifunctor of the Ext$^{1}$ bifunctor on $\mathscr{A}$, in the sense of \cite[Definition 3.6]{herschend2017n}. Our alternative perspective also leads us to a proof of the main theorem which is quite different to the proof presented in \cite{BalSch} for the triangulated case and \cite{wwzz} for extriangulated case. In our work, the role of the idempotent morphisms is clarified and the extriangles of the idempotent completion (and weak idempotent completion) have an explicit description which isn't available in the treatment of \cite{BalSch} and \cite{wwzz}. As a consequence, with some additional work, we can prove that the weak idempotent completion is also extriangulated as a corollary to the fact that the idempotent completion is extriangulated. We do so by showing that the weak idempotent completion is an extension-closed subcategory of the idempotent completion; to the best of our knowledge, this additional result been shown in the triangulated or exact or extriangulated case.

This paper is organised as follows: in \S 2, we recall the necessary theory of idempotent completions, weak idempotent completions and the theory of extriangulated categories. Finally, in \S 3, we show that the idempotent completion and weak idempotent completion of an extriangulated category are also extriangulated. 

\section{Idempotent completion and extriangulated categories.}
In this section, we recall the basic theory of idempotent completions of additive categories, weak idempotent completions of additive categories and the theory extriangulated categories as introduced in \cite{NakaokaPalu}. 

Let us set the common notation for this section. Let $\mathscr{A}$ be an additive category. Given objects $X,Y$ in $\mathscr{A}$ we will write $\mathscr{A}(X,Y)$ for the group of morphisms $X \rightarrow Y$. For an object $X$ in $\mathscr{A}$ we denote the identity morphism of $X$ by $1_X$. 
\subsection{Idempotent completeness and weakly idempotent completeness.}

\begin{definition}\label{IdemComp}\cite[Definition 1.21,1.2.2]{karoubi}.
Let $\mathscr{A}$ be an additive category. We say that $\mathscr{A}$ is idempotent complete if for every idempotent morphism $p \colon A \rightarrow A$ ( i.e. $p^2=p$) in $\mathscr{A}$, there is a decomposition $A \cong K \oplus I$ of $A$ such that $p \cong \begin{pmatrix}
0 & 0\\
0 & 1_{I}
\end{pmatrix}$ with respect to this decomposition.
\end{definition}

\begin{proposition}\cite[Remark 6.2]{buhler}.
An additive category $\mathscr{A}$ is idempotent complete if and only if every idempotent morphism admits a kernel.
\end{proposition}

Every additive category $\mathscr{A}$ embeds fully faithfully into an idempotent complete category $\tilde{\mathscr{A}}$. The category $\tilde{\mathscr{A}}$ is commonly referred to as the \textit{idempotent completion} or as the \textit{Karoubi envelope} of $\mathscr{A}$. 

\begin{definition}\cite[1.2 Definition]{BalSch}.
Let $\mathscr{A}$ be an additive category. The \textit{idempotent completion} of $\mathscr{A}$ is denoted by $\tilde{\mathscr{A}}$  and is defined as follows. The objects of $\tilde{\mathscr{A}}$ are the pairs $(A,p)$ where $A$ is an object of $\mathscr{A}$ and $p \colon A \rightarrow A$ is an idempotent morphism. A morphism in $\tilde{\mathscr{A}}$ from $(A,p)$ to $(B,q)$ is a morphism $\sigma \colon A \rightarrow B \in \mathscr{A}$ such that $\sigma p = q \sigma = \sigma$. For any object $(A,p)$ in $\tilde{\mathscr{A}}$, the identity morphism $1_{(A,p)} = p$. 
\end{definition}

\begin{proposition}\label{karoubian}\cite[See e.g. Remark 6.3]{buhler}.
Let $\mathscr{A}$ be an additive category. The Karoubi envelope $\tilde{\mathscr{A}}$ is an idempotent complete category. The biproduct in $\tilde{\mathscr{A}}$ is defined as $(A,p) \oplus (B,q) = (A \oplus B, p \oplus q)$. There is a fully faithful additive functor $i_{\mathscr{A}} \colon \mathscr{A} \rightarrow \tilde{\mathscr{A}}$ defined as follows. For an object $A$ in $\mathscr{A}$, we have that $i_{\mathscr{A}}(A) = (A,1_A)$ and for a morphism $f$ in $\mathscr{A}$, we have that $i_{\mathscr{A}}(f)=f$.
\end{proposition}

The Karoubi envelope is unique with respect to the following universal property. 
\begin{proposition}\cite[Proposition 6.10]{buhler}.
Let $\mathscr{A}$ be an additive category and let $\mathscr{B}$ be an idempotent complete category. For every additive functor $F \colon \mathscr{A} \rightarrow \mathscr{B}$, there exists a functor $\tilde{F} \colon \tilde{\mathscr{A}} \rightarrow \mathscr{B}$ and a natural isomorphism $\alpha \colon F \Rightarrow \tilde{F}i_{\mathscr{A}}.$
\end{proposition}

We now introduce the related notion of a \textit{weakly idempotent complete} category. To do, we must first recall the following definition which stems from work by Thomason on exact categories with \textit{weakly split idempotents}, see \cite[A.5.1.]{thomason}.

\begin{definition}\cite[\S 7]{buhler}. Let $\mathscr{A}$ be an arbitrary category. A morphism $r \colon B \rightarrow C$ is called a \textit{retraction} if there exists a morphism $q \colon C \rightarrow B$ such that $rq=1_C$. A morphism $s \colon A \rightarrow B$ is called a \textit{section} if there exists a morphism $t \colon B \rightarrow A$ such that $ts=1_A$. 
\end{definition}
If $r \colon B \rightarrow C$ is a retraction with a section $s \colon C \rightarrow B$ then the composition $sr$ is an idempotent morphism. This idempotent gives a decomposition of $B$ in the sense of Definition \ref{IdemComp} if the morphism $r$ admits a kernel. See \cite[Remark 7.4]{buhler} for more details. 

\begin{proposition}\cite[Lemma 7.1]{buhler}. Let $\mathscr{A}$ be an additive category. Then the following statements are equivalent:
\begin{enumerate}
\item Every section has a cokernel.
\item Every retraction has a kernel.
\end{enumerate}
\end{proposition}

\begin{definition}\cite[Definition 7.2]{buhler}.
Let $\mathscr{A}$ be an additive category. Then $\mathscr{A}$ is said to be \textit{weakly idempotent complete} if every retraction has a kernel. Equivalently, $\mathscr{A}$ is weakly idempotent complete if every section has a cokernel. 
\end{definition}

Every \textit{small} additive category $\mathscr{A}$ embeds fully faithfully into a weakly idempotent complete category $\hat{\mathscr{A}}$. We call the category $\hat{\mathscr{A}}$ a weak idempotent completion of $\mathscr{A}$. The construction of the weak idempotent completion $\hat{\mathscr{A}}$ is similar to that of the idempotent completion $\tilde{\mathscr{A}}$.

\begin{definition}\cite[Definition 3.1]{selinger}. Let $\mathscr{A}$ be any category and $A$ an object in $\mathscr{C}$. An idempotent morphism $e \colon A \rightarrow A$ is said to \textit{split} if it admits a retraction $r \colon A \rightarrow X$ and a section $s \colon X \rightarrow A$ such that $s \circ r = e$ and $r \circ s = 1_{X}$. 
\end{definition}

\begin{definition}\label{weaklyIdemComp}\cite[1.12]{Neeman}.
Let $\mathscr{A}$ be a small additive category. The \textit{weak idempotent completion} of $\mathscr{A}$ is denoted by $\hat{\mathscr{A}}$ and is defined as follows. The objects of $\hat{\mathscr{A}}$ are the pairs $(A,p)$ where $A$ is an object of $\mathscr{A}$ and $p \colon A \rightarrow A$ is an idempotent morphism that splits. A morphism in $\hat{\mathscr{A}}$ from $(A,p)$ to $(B,q)$ is a morphism $\sigma \colon A \rightarrow B \in \mathscr{A}$ such that $\sigma p = q \sigma = \sigma$.  For any object $(A,p)$ in $\hat{\mathscr{A}}$, the identity morphism $1_{(A,p)}=p.$
\end{definition}

\begin{proposition}\label{neeman}\cite[E.g Remark 7.8]{buhler}.
Let $\mathscr{A}$ be a small additive category. The weak idempotent completion $\hat{\mathscr{A}}$ is weakly idempotent complete. The biproduct in $\hat{\mathscr{A}}$ is defined as $(A,p) \oplus (B,q) = (A \oplus B, p \oplus q)$. There is a fully faithful additive functor $j_{\mathscr{A}} \colon \mathscr{A} \rightarrow \hat{\mathscr{A}}$ defined as follows. For an object $A$ in $\mathscr{A}$, we have that $j_{\mathscr{A}}(A) = (A,1_A)$ and for a morphism $f$ in $\mathscr{A}$, we have that $j_{\mathscr{A}}(f)=f$.
\end{proposition}

\begin{proposition}\cite[Remark 7.8]{buhler}.
Let $\mathscr{A}$ be a small additive category and let $\mathscr{B}$ be a weakly idempotent complete category. For every additive functor $F \colon \mathscr{A} \rightarrow \mathscr{B}$, there exists a functor $\hat{F} \colon \hat{\mathscr{A}} \rightarrow \mathscr{B}$ and a natural isomorphism $\alpha \colon F \Rightarrow \hat{F}j_{\mathscr{A}}.$
\end{proposition}
The reader is directed to \S 6 and \S 7 of \cite{buhler} for a more extensive exposition of idempotent completeness and weakly idempotent completeness.

\subsection{Extriangulated categories.}
In this section, we will recall mostly from \cite{NakaokaPalu} the basic theory of extriangulated categories needed for this paper. Through out this subsection, $\mathscr{C}$ will be an additive category equipped with a biadditive functor $\mathbb{E} \colon \mathscr{C}^{\text{op}} \times \mathscr{C} \rightarrow Ab$, where $Ab$ is the category of abelian groups. 

\begin{definition}\cite[Definition 2.1]{NakaokaPalu}.
Let $A,C$ be objects of $\mathscr{C}$. An element $\d \in \mathbb{E}(C,A)$ is called an $\mathbb{E}$-extension. Formally, an $\mathbb{E}$-extension is a triple $(A,\d,C)$. 

Since $\mathbb{E}$ is a bifunctor, for any $a \in \mathscr{C}(A,A^{\prime})$ and $c \in \mathscr{C}(C^{\prime},C)$, we have the following $\mathbb{E}$-extensions:
$$a_{*}\d := \mathbb{E}(C,a)(\d) \in \mathbb{E}(C,A^{\prime}),$$
$$c^{*}\d := \mathbb{E}(c^{\text{op}},A)(\d) \in \mathbb{E}(C^{\prime},A) \text{ and }$$
$$c^{*}a_{*}\d = a_{*}c^{*}\d := \mathbb{E}(c^{\text{op}},a)(\d) \in \mathbb{E}(C^{\prime},A^{\prime}).$$
\end{definition}
We will abuse notation by writing $\mathbb{E}(c,-)$ instead of $\mathbb{E}(c^{\text{op}},-).$

\begin{definition}\cite[Definition 2.3]{NakaokaPalu}.
Let $(A,\d,C)$ and $(A^{\prime}, \d^{\prime}, C^{\prime})$ be any pair of $\mathbb{E}$-extensions. A morphism $(a,c) \colon \d \rightarrow \d^{\prime}$ of $\mathbb{E}$-extensions is a pair of morphisms $a \in \mathscr{C}(A,A^{\prime})$ and $c \in \mathscr{C}(C,C^{\prime})$ such that: $$ a_{*}\d = c^{*}\d^{\prime}.$$ 
\end{definition}

\begin{lemma}\label{*operation}\cite[Remark 2.4]{NakaokaPalu}.
Let $(A,\d,C)$ be an $\mathbb{E}$-extension. Then we have the following.
\begin{enumerate}
\item Any morphism $a \in \mathscr{C}(A,A^{\prime})$ induces a morphism of $\mathbb{E}$-extensions,
$$(a,1_C) \colon \d \rightarrow a_{*}\d.$$
\item Any morphism $c \in \mathscr{C}(C^{\prime},C)$ induces a morphism of $\mathbb{E}$-extensions,
$$(1_A,c) \colon c^{*}\d \rightarrow \d.$$
\end{enumerate}
\end{lemma}

\begin{definition}\cite[Definition 2.5]{NakaokaPalu}.
For any objects $A,C$ in $\mathscr{C}$, the zero element $0 \in \mathbb{E}(C,A)$ is called a \textit{split $\mathbb{E}$-extension}. 
\end{definition}

\begin{definition}\cite[Definition 2.6]{NakaokaPalu}.
Let $\d \in \mathbb{E}(C,A)$ and $\d^{\prime} \in \mathbb{E}(C^{\prime}, A^{\prime})$ be any pair of $\mathbb{E}$-extensions. Let $i_{C} \colon C \rightarrow C \oplus C^{\prime}$ and $i_{C^{\prime}} \colon C^{\prime} \rightarrow C \oplus C^{\prime}$ be the canonical inclusion maps. Let $p_{A} \colon A \oplus A^{\prime} \rightarrow A$,  and $p_{A^{\prime}} \colon A \oplus A^{\prime} \rightarrow A^{\prime}$ be the canonical projection maps. By the biadditivity of $\mathbb{E}$ we have the following isomorphism.
$$ \mathbb{E}(C \oplus C^{\prime}, A \oplus A^{\prime}) \cong \mathbb{E}(C,A) \oplus \mathbb{E}(C,A^{\prime}) \oplus \mathbb{E}(C^{\prime}, A) \oplus \mathbb{E}(C^{\prime}, A^{\prime})$$ 

Let $\d \oplus \d^{\prime} \in \mathbb{E}(C \oplus C^{\prime}, A \oplus A^{\prime})$ be the element corresponding to $(\d,0,0,\d^{\prime})$ via the above isomorphism. 
If $A = A^{\prime}$ and $C = C^{\prime}$, then the sum $\d + \d^{\prime} \in \mathbb{E}(C,A)$ is obtained by 
$$ \d + \d^{\prime} = \mathbb{E}(\Delta_{C},\nabla_{A})(\d \oplus \d^{\prime}),$$
where $\Delta_{C} = \begin{pmatrix}
1 \\ 1
\end{pmatrix} : C \rightarrow C \oplus C$, and $\nabla_{A} = \begin{pmatrix}
1,1
\end{pmatrix}: A \oplus A \rightarrow A.$
\end{definition}

\begin{definition}\label{equivReln}\cite[Definition 2.7]{NakaokaPalu}.
Let $A,C$ be a pair of objects in $\mathscr{C}$. Two sequences of morphisms $A \overset{x}{\longrightarrow} B \overset{y}{\longrightarrow} C$, and $A \overset{x^{\prime}}{\longrightarrow} B^{\prime} \overset{y^{\prime}}{\longrightarrow} C$ in $\mathscr{C}$ are said to be \textit{equivalent} if there exists an isomorphism $b \in \mathscr{C}(B,B^{\prime})$ such that the following diagram commutes. 
\begin{center}
\begin{tikzcd}
A \arrow[r, "x"] \arrow[d, equal]
& B \arrow[d, "b"] \arrow[r,"y"]
& C \arrow[d, equal] \\
 A \arrow[r,"x^{\prime}"]
&B^{\prime} \arrow[r,"y^{\prime}"]
&C \end{tikzcd}
\end{center}
We denote the equivalence class of a sequence  $A \overset{x}{\longrightarrow} B \overset{y}{\longrightarrow} C$, by 
$[A \overset{x}{\longrightarrow} B \overset{y}{\longrightarrow} C]$.
\end{definition}

\begin{definition}\cite[Definition 2.8]{NakaokaPalu}.
Let $A,B,C,A^{\prime}, B^{\prime}, C^{\prime}$ be objects in the category $\mathscr{C}$. 
\begin{enumerate}
\item We denote by $0$ the equivalence class $[A \overset{\big[\begin{smallmatrix}
1_A\\
0
\end{smallmatrix}\big]}{\longrightarrow} A \oplus C \overset{[\begin{smallmatrix} 0 & 1_C \end{smallmatrix}]}{\longrightarrow} C]$.

\item For any two equivalence classes  $[A \overset{x}{\longrightarrow} B \overset{y}{\longrightarrow} C]$ and $[A^{\prime} \overset{x^{\prime}}{\longrightarrow} B^{\prime} \overset{y^{\prime}}{\longrightarrow} C^{\prime}]$, we denote by $[A \overset{x}{\longrightarrow} B \overset{y}{\longrightarrow} C] \oplus [A^{\prime} \overset{x^{\prime}}{\longrightarrow} B^{\prime} \overset{y^{\prime}}{\longrightarrow} C^{\prime}]$ the equivalence class
$[A \oplus A^{\prime} \overset{x\oplus x^{\prime}}{\longrightarrow} B \oplus B^{\prime} \overset{y \oplus y^{\prime}}{\longrightarrow} C \oplus C^{\prime}].$
\end{enumerate}
\end{definition}

\begin{definition}\cite[Definition 2.9]{NakaokaPalu}.
Let $\mathfrak{s}$ be a correspondence associating an equivalence class $\mathfrak{s}(\d) = [A \overset{x}{\longrightarrow} B \overset{y}{\longrightarrow} C]$ to any $\mathbb{E}$-extension $\d \in \mathbb{E}(C,A)$.  We say that $\mathfrak{s}$ is a \textit{realisation} of $\mathbb{E}$ if the following condition $(\circ)$ holds. 

$(\circ)$ Let $\d \in \mathbb{E}(C,A)$ and $\d^{\prime} \in \mathbb{E}(C^{\prime},A^{\prime})$ be $\mathbb{E}$-extensions, with $\mathfrak{s}(\d) = [A \overset{x}{\longrightarrow} B \overset{y}{\longrightarrow} C]$ and $\mathfrak{s}(\d^{\prime}) = [A^{\prime} \overset{x^{\prime}}{\longrightarrow} B^{\prime} \overset{y^{\prime}}{\longrightarrow} C^{\prime}]$. Then for any morphism $(a,c) \colon \d \rightarrow \d^{\prime}$ of $\mathbb{E}$-extensions, there exists $b \in \mathscr{C}(B,B^{\prime})$ such that the following diagram commutes. 
\begin{center}
\begin{tikzcd}
A \arrow[r, "x"] \arrow[d, "a"]
& B \arrow[d, dashed, "b"] \arrow[r,"y"]
& C \arrow[d, "c"] \\
 A^{\prime} \arrow[r,"x^{\prime}"]
&B^{\prime} \arrow[r,"y^{\prime}"]
&C^{\prime} \end{tikzcd}
\end{center}
In this situation, we say that the triple of morphisms $(a,b,c)$ realises $(a,c)$. For $\d \in \mathbb{E}(C,A)$, we say that the sequence $A \overset{x}{\longrightarrow} B \overset{y}{\longrightarrow} C$ realises $\d$ if $\mathfrak{s}(\d)=[A \overset{x}{\longrightarrow} B \overset{y}{\longrightarrow} C]$.
\end{definition}

\begin{definition}\cite[Definition 2.10]{NakaokaPalu}.
A realisation $\mathfrak{s}$ is said to be an \textit{additive realisation} if the following conditions are satisfied,
\begin{enumerate}
\item For any objects $A,C$ in $\mathscr{C}$, the split $\mathbb{E}$-extension $ 0 \in \mathbb{E}(C,A)$ satisfies $$\mathfrak{s}(0)=0.$$
\item For any pair of extensions $\d$ and $\d^{\prime}$, we have that, $$\mathfrak{s}(\d \oplus \d^{\prime}) = \mathfrak{s}(\d) \oplus \mathfrak{s}(\d^{\prime}).$$
\end{enumerate} 
\end{definition}

We are now in a position to define an extriangulated category.

\begin{definition}\label{DefExtriang}\cite[Definition 2.12]{NakaokaPalu}. Let $\mathscr{C}$ be an additive category. An \textit{extriangulated category} is a triple $(\mathscr{C},\mathbb{E},\mathfrak{s})$ satisfying the following axioms.

\begin{itemize}

\item[] (ET1) The functor $\mathbb{E} \colon \mathscr{C}^{\text{op}} \times \mathscr{C} \rightarrow Ab \text{ is a biadditive functor}$. 

\item[] (ET2) The correspondence $\mathfrak{s}$ is an additive realisation of $\mathbb{E}$. 

\item[] (ET3) Let $\d \in \mathbb{E}(C,A)$ and $\d^{\prime} \in \mathbb{E}(C^{\prime},A^{\prime})$ be any pair of $\mathbb{E}$-extensions, realised as  $\mathfrak{s}(\d) = [A \overset{x}{\longrightarrow} B \overset{y}{\longrightarrow} C]$ and $\mathfrak{s}(\d^{\prime}) = [A^{\prime} \overset{x^{\prime}}{\longrightarrow} B^{\prime} \overset{y^{\prime}}{\longrightarrow} C^{\prime}]$ respectively. For any commutative diagram 
\begin{center}
\begin{tikzcd}
A \arrow[r, "x"] \arrow[d, "a"]
& B \arrow[d, "b"] \arrow[r,"y"]
& C \\
 A^{\prime} \arrow[r,"x^{\prime}"]
&B^{\prime} \arrow[r,"y^{\prime}"]
&C^{\prime} \end{tikzcd}
\end{center}
there exists a morphism $c \in \mathscr{C}(C,C^{\prime})$ such that $(a,c) \colon \d \rightarrow \d^{\prime}$ is a morphism of $\mathbb{E}$-extensions and the triple $(a,b,c)$ realises $(a,c)$.

\item[] $\text{(ET3})^{\text{op}}$ The dual of (ET3).

\item[] (ET4) Let $\d \in \mathbb{E}(D,A)$ and $\d^{\prime} \in \mathbb{E}(F,B)$ be any pair of $\mathbb{E}$-extensions, realised by the sequences, $A \overset{f}{\longrightarrow} B \overset{f^{\prime}}{\longrightarrow} D$ and $B\overset{g}{\longrightarrow} C \overset{g^{\prime}}{\longrightarrow} F$. Then there exists an object $E$ in $\mathscr{C}$, a commutative diagram
\begin{center}
\begin{tikzcd}
A \arrow[r, "f"] \arrow[d, equal]
& B \arrow[d, "g"] \arrow[r,"f^{\prime}"]
& D \arrow[d,"d"] \\
 A \arrow[r,"h"]
&C \arrow[r,"h^{\prime}"] \arrow[d,"g^{\prime}"]
&E \arrow[d,"e"] \\
 & F \arrow[r,equal] & F
\end{tikzcd}
\end{center}
in $\mathscr{C}$ and an $\mathbb{E}$-extension $\d^{\prime \prime} \in \mathbb{E}(E,A)$ realised by the sequence $A \overset{h}{\longrightarrow} C \overset{h^{\prime}}{\longrightarrow} E$, which satisfy the following compatibilities:
\begin{enumerate}[(i)]
\item $\mathfrak{s}((f^{\prime})_{*}\d^{\prime}) = [D \overset{d}{\longrightarrow} E \overset{e}{\longrightarrow} F].$

\item $d^{*} \d^{\prime \prime} = \d.$

\item $f_{*}\d^{\prime \prime} = e^{*}\d^{\prime}.$
\end{enumerate}

\item[] $\text{(ET4)}^{\text{op}}$ The dual of (ET4). 
\end{itemize}
In this case, we call $\mathfrak{s}$ an $\mathbb{E}$-triangulation of $\mathscr{C}$.
\end{definition}
There are many examples of extriangulated categories. They include, exact categories, and triangulated categories and extension-closed subcategories of triangulated subcategories. There are also extriangulated categories which are neither exact nor triangulated; for example see, \cite[Proposition 3.30]{NakaokaPalu}, \cite[Example 4.14 and Corollary 4.12]{ZhouZhu}.
\begin{definition}\cite[Definition 2.31]{BTShah}.
Let $(\mathscr{C},\mathbb{E},\mathfrak{s})$ and $(\mathscr{C}^{\prime},\mathbb{E}^{\prime},\mathfrak{s}^{\prime})$ be extriangulated categories. A covariant additive functor $F \colon \mathscr{C} \rightarrow \mathscr{C}^{\prime}$ is called an \textit{extriangulated functor} if there exists a natural transformation $$\Gamma = \{\Gamma_{(C,A)}\}_{(C,A) \in \mathscr{C}^{\text{op}} \times \mathscr{C}} \colon \mathbb{E} \Rightarrow \mathbb{E}^{\prime}(F^{\text{op}}-,F-)$$ 
of functors $\mathscr{C}^{\text{op}} \times \mathscr{C} \rightarrow Ab$, such that $\mathfrak{s}(\d) = [X \overset{x}{\longrightarrow} Y \overset{y}{\longrightarrow} Z]$ implies that $\mathfrak{s}^{\prime}(\Gamma_{(Z,X)})(\d) = [F(A) \overset{F(x)}{\longrightarrow} F(B) \overset{F(y)}{\longrightarrow} F(C)]$. Here $F^{\text{op}}$ is the \textit{opposite functor} $\mathscr{C}^{\text{op}} \rightarrow \mathscr{C^{\prime}}^{\text{op}}$ given by $F^{\text{op}}(A)=F(A)$ and $F^{\text{op}}(f^{\text{op}})=(F(f))^{\text{op}}$. Furthermore, we say that $F$ is an \textit{extriangulated equivalence} if $F$ is an equivalence of categories.  
\end{definition}
We will conclude this section by introducing some useful terminology from \cite{NakaokaPalu} and stating results about extriangulated categories which will be helpful for the rest of the paper. 

\begin{definition}\cite[Definition 2.5, Definition 3.9]{NakaokaPalu}.
Let $(\mathscr{C},\mathbb{E},\mathfrak{s})$ be a triple satisfying (ET1), (ET2), (ET3) and $\text{(ET3)}^{\text{op}}$. 
\begin{enumerate}
\item A sequence $A \overset{x}{\longrightarrow} B \overset{y}{\longrightarrow} C$ is called a \textit{conflation} if it realises some $\mathbb{E}$-extension $\d \in \mathbb{E}(C,A)$.
\item A morphism $f \in \mathscr{C}(A,B)$ is called an \textit{inflation} if it admits some conflation $A \overset{f}{\longrightarrow} B \longrightarrow C$. 
\item A morphism $g \in \mathscr{C}(B,C)$ is called a \textit{deflation} if it admits some conflation $A \longrightarrow B \overset{g}{\longrightarrow} C$. 
\end{enumerate}
\end{definition}
The terminology of conflations, inflations and deflations is also used in the context of exact categories and triangulated categories analogously. 

\begin{definition}\cite[Definition 2.19]{NakaokaPalu}. Let $(\mathscr{C},\mathbb{E},\mathfrak{s})$ be a triple satisfying (ET1), (ET2). 
\begin{enumerate}
\item If a conflation $A \overset{x}{\longrightarrow} B \overset{y}{\longrightarrow} C$ realises $\d \in \mathbb{E}(C,A)$, we call the pair $(A \overset{x}{\longrightarrow} B \overset{y}{\longrightarrow} C, \d)$ an $\mathbb{E}$-\textit{triangle} or \textit{extriangle} and denote it by the following diagram.
\begin{center}
\begin{tikzcd}
A \arrow[r, "x"]
& B \arrow[r,"y"]
& C \arrow[r, dashed, "\d"] & \text{}
\end{tikzcd} 
\end{center}

\item Let \begin{tikzcd} 
A \arrow[r, "x"]
& B \arrow[r,"y"]
& C \arrow[r, dashed, "\d"] &\text{}
\end{tikzcd}and
\begin{tikzcd} 
A^{\prime} \arrow[r, "x^{\prime}"]
& B^{\prime} \arrow[r,"y^{\prime}"]
& C^{\prime} \arrow[r, dashed, "\d^{\prime}"] &\text{}
\end{tikzcd}
be any pair of $\mathbb{E}$-triangles. If a triple $(a,b,c)$ realises $(a,c) \colon \d \rightarrow \d^{\prime}$ we write it as in the following commutative diagram and call $(a,b,c)$ a morphism of $\mathbb{E}$-triangles. 
\begin{center}
\begin{tikzcd}
A \arrow[r, "x"] \arrow[d,"a"]
& B \arrow[r,"y"] \arrow[d,"b"]
& C \arrow[r, dashed, "\d"] \arrow[d,"c"] & \text{} \\
A^{\prime} \arrow[r, "x^{\prime}"]
& B^{\prime} \arrow[r,"y^{\prime}"]
& C^{\prime} \arrow[r, dashed, "\d^{\prime}"] &\text{}
\end{tikzcd}
\end{center}

\end{enumerate}
\end{definition}

\begin{lemma}\label{2out3iso}\cite[Corollary 3.6]{NakaokaPalu}. 
Let $(\mathscr{C},\mathbb{E},\mathfrak{s})$ be a triple satisfying (ET1), (ET2), (ET3) and $\text{(ET3)}^{\text{op}}$. Let $(a,b,c)$ be a morphism of $\mathbb{E}$-triangles. If any two of $a,b,c,$ are isomorphisms, then so is the third. 
\end{lemma}

\begin{lemma}\label{closedUnderIso1}\cite[Proposition 3.7]{NakaokaPalu}.
Let $(\mathscr{C},\mathbb{E},\mathfrak{s})$ be a triple satisfying (ET1), (ET2), (ET3) and $\text{(ET3)}^{\text{op}}$. Let \begin{center}
\begin{tikzcd}
A \arrow[r, "a"]
& B \arrow[r,"b"]
& C \arrow[r, dashed, "\d"] & \text{}
\end{tikzcd} 
\end{center}
be any $\mathbb{E}$-triangle in $\mathscr{C}$. If $f \in \mathscr{C}(A,X)$ and $h \in \mathscr{C}(C,Z)$ are isomorphisms, then 
\begin{center}
\begin{tikzcd}
X \arrow[r, "a \circ f^{-1}"]
& B \arrow[r,"h \circ b"]
& Z \arrow[r, dashed, "f_{*}h^{*}\d"] & \text{}
\end{tikzcd} 
\end{center}
is again an $\mathbb{E}$-triangle. 

\end{lemma}

\begin{corollary}\label{closedUnderIso2}
Let $(\mathscr{C},\mathbb{E},\mathfrak{s})$ be a triple satisfying (ET1), (ET2), (ET3) and $\text{(ET3)}^{\text{op}}$. Let \begin{center}
\begin{tikzcd}
A \arrow[r, "a"]
& B \arrow[r,"b"]
& C \arrow[r, dashed, "\d"] & \text{}
\end{tikzcd} 
\end{center}
be any $\mathbb{E}$-triangle in $\mathscr{C}$. Suppose we have the following commutative diagram,
\begin{center}
\begin{tikzcd}
A \arrow[r, "a"] \arrow[d,"f"]
& B \arrow[r,"b"] \arrow[d,"g"]
& C \arrow[d,"h"]
\\
X \arrow[r,"x"]
& Y \arrow[r,"y"]
& Z 
\end{tikzcd} 
\end{center}
where the morphisms $f,g,h$ are isomorphisms. Then it follows that
\begin{center}
\begin{tikzcd}
X \arrow[r, "x"]
& Y \arrow[r,"y"]
& Z \arrow[r, dashed, "f_{*}h^{*}\d"] & \text{}
\end{tikzcd} 
\end{center}
is an $\mathbb{E}$-triangle. 
\begin{proof}
By Proposition \ref{closedUnderIso1}, \begin{center}
\begin{tikzcd}
X \arrow[r, "a \circ f^{-1}"]
& B \arrow[r,"h \circ b"]
& Z \arrow[r, dashed, "f_{*}h^{*}\d"] & \text{}
\end{tikzcd} 
\end{center}
is an $\mathbb{E}$-triangle. Now consider the following diagram. 
\begin{center}
\begin{tikzcd}
X \arrow[r, "a \circ f^{-1}"] \arrow[d,equal]
& B \arrow[r,"h \circ b"] \arrow[d,"g"]
& Z \arrow[d,equal]
\\
X \arrow[r,"x"]
& Y \arrow[r,"y"]
& Z 
\end{tikzcd} 
\end{center}
Observe that it commutes, therefore it is an equivalence, which implies that 
\begin{center}
\begin{tikzcd}
X \arrow[r, "x"]
& Y \arrow[r,"y"]
& Z \arrow[r, dashed, "f_{*}h^{*}\d"] & \text{}
\end{tikzcd} 
\end{center}
is an $\mathbb{E}$-triangle. 
\end{proof}
\end{corollary}

The following two propositions are special cases of propositions from \cite{herschend2017n} which are stated for general $n$-exangulated categories, but here we are restating them in the case of extriangulated categories which are in fact the same as $1$-exangulated categories by \cite[Proposition 4.3]{herschend2017n}. 

\begin{proposition}\cite[Proposition 3.2]{herschend2017n}. Let $(\mathscr{C},\mathbb{E},\mathfrak{s})$ be a triple satisfying (ET1), (ET2), (ET3) and $\text{(ET3)}^{\text{op}}$. Let $(A \overset{a}{\longrightarrow} B \overset{b}{\longrightarrow} C, \d)$ and $(X \overset{x}{\longrightarrow} Y \overset{y}{\longrightarrow} Z, \rho)$ be pairs consisting of a sequence of morphisms and an $\mathbb{E}$-extension. Then the following statements are equivalent.

\begin{enumerate}
\item $(A \oplus X \overset{a \oplus x}{\longrightarrow} B \oplus Y \overset{b \oplus y}{\longrightarrow} C \oplus Z, \d \oplus \rho)$ is an $\mathbb{E}$-triangle.
\item Both $(A \overset{a}{\longrightarrow} B \overset{b}{\longrightarrow} C, \d)$ and $(X \overset{x}{\longrightarrow} Y \overset{y}{\longrightarrow} Z, \rho)$ are $\mathbb{E}$-triangles. 
\end{enumerate}
\end{proposition}

\begin{proposition}\label{dirSum}\cite[Corollary 3.3]{herschend2017n}.
Let $(\mathscr{C},\mathbb{E},\mathfrak{s})$ be a triple satisfying (ET1), (ET2), (ET3) and $\text{(ET3)}^{\text{op}}$. Suppose that 
\begin{center}
\begin{tikzcd}
X \oplus A \arrow[r, "\begin{pmatrix}
x \hspace{0.3cm} u \\
v \hspace{0.3cm} 1
\end{pmatrix}"]
& Y \oplus A \arrow[r,"\begin{pmatrix}
y \hspace{0.3cm} w
\end{pmatrix}"]
& Z \arrow[r, dashed, "\d"] & \text{}
\end{tikzcd} 
\end{center}
is an $\mathbb{E}$-triangle. Then for $t = x - u \circ v$ and $p = [1,0] \colon X \oplus A \rightarrow X$,

\begin{center}
\begin{tikzcd}
X \arrow[r, "t"]
& Y \arrow[r,"y"]
& Z \arrow[r, dashed, "p_{*}\d"] & \text{}
\end{tikzcd} 
\end{center}
is an $\mathbb{E}$-triangle.
\end{proposition}

\begin{proposition}\label{longExact}\cite[Corollary 3.12]{NakaokaPalu}.
Let $(\mathscr{C},\mathbb{E},\mathfrak{s})$ be an extriangulated category. For any $\mathbb{E}$-triangle \begin{tikzcd} 
A \arrow[r, "x"]
& B \arrow[r,"y"]
& C \arrow[r, dashed, "\d"] &\text{,}
\end{tikzcd}the following sequences of natural transformations are exact. 

\begin{center}
\begin{tikzcd}
\mathscr{C}(C,-) \arrow[r, Rightarrow, "\mathscr{C}(y{,}-)"]
& \mathscr{C}(B,-) \arrow[r, Rightarrow, "\mathscr{C}(x{,}-)"]
& \mathscr{C}(A,-) \arrow[r, Rightarrow, "\d^{\#}"]
& \mathbb{E}(C,-) \arrow[r, Rightarrow, "\mathbb{E}(y{,}-)"]
& \mathbb{E}(B,-) \arrow[r, Rightarrow, "\mathbb{E}(x{,}-)"]
& \mathbb{E}(A,-)
\end{tikzcd}
\end{center}

\begin{center}
\begin{tikzcd}
\mathscr{C}(-,A) \arrow[r, Rightarrow, "\mathscr{C}(-{,}x)"]
& \mathscr{C}(-,B) \arrow[r, Rightarrow, "\mathscr{C}(-{,}y)"]
& \mathscr{C}(-,C) \arrow[r, Rightarrow, "\d_{\#}"]
& \mathbb{E}(-,A) \arrow[r, Rightarrow, "\mathbb{E}(-{,}x)"]
& \mathbb{E}(-,B) \arrow[r, Rightarrow, "\mathbb{E}(-{,}y)"]
& \mathbb{E}(-,C)
\end{tikzcd}
\end{center}
The natural transformations $\d^{\#}$ and $\d_{\#}$ are defined as follows. Given any object $X$ in $\mathscr{C}$, we have that
\begin{enumerate}
\item $(\d^{\#})_{X} \colon \mathscr{C}(A,X) \rightarrow \mathbb{E}(C,X) \text{ ; } g \mapsto f_{*}\d,$

\item $(\d_{\#})_{X} \colon \mathscr{C}(X,C) \rightarrow \mathbb{E}(X,A) \text{ ; } f \mapsto f^{*}\d.$
\end{enumerate}

The exactness of the first sequence of natural transformations means that for any object $X$ in $\mathscr{C}$, the sequence
\begin{center}
\begin{tikzcd}
\mathscr{C}(C,X) \arrow[r, "\mathscr{C}(y{,}X)"]
& \mathscr{C}(B,X) \arrow[r, "\mathscr{C}(x{,}X)"]
& \mathscr{C}(A,X) \arrow[r, "\d^{\#}_{X}"]
& \mathbb{E}(C,X) \arrow[r, "\mathbb{E}(y{,}X)"]
& \mathbb{E}(B,X) \arrow[r,"\mathbb{E}(x{,}X)"]
& \mathbb{E}(A,X)
\end{tikzcd}
\end{center}
is exact in $Ab$ and likewise for the second sequence.
\end{proposition}

\begin{proposition}\label{et3Equiv}\cite[Proposition 3.3]{NakaokaPalu}.
Let $(\mathscr{C},\mathbb{E},\mathfrak{s})$ be a triple satisfying (ET1) and (ET2). Then the following are equivalent. 
\begin{enumerate}
\item $(\mathscr{C},\mathbb{E},\mathfrak{s})$ satisfies (ET3) and (ET3)$^{\text{op}}$.
\item For any $\mathbb{E}$-triangle \begin{tikzcd}
A \arrow[r,"x"]
& B \arrow[r,"y"]
& C \arrow[r,dashed,"\d"]
&\text{}
\end{tikzcd}, the following sequences of natural transformations are exact. 
\begin{center}
\begin{tikzcd}
\mathscr{C}(C,-) \arrow[r, Rightarrow, "\mathscr{C}(y{,}-)"]
& \mathscr{C}(B,-) \arrow[r, Rightarrow, "\mathscr{C}(x{,}-)"]
& \mathscr{C}(A,-) \arrow[r, Rightarrow, "\d^{\#}"]
& \mathbb{E}(C,-) \arrow[r, Rightarrow, "\mathbb{E}(y{,}-)"]
& \mathbb{E}(B,-) 
\end{tikzcd}
\end{center}

\begin{center}
\begin{tikzcd}
\mathscr{C}(-,A) \arrow[r, Rightarrow, "\mathscr{C}(-{,}x)"]
& \mathscr{C}(-,B) \arrow[r, Rightarrow, "\mathscr{C}(-{,}y"]
& \mathscr{C}(-,C) \arrow[r, Rightarrow, "\d_{\#}"]
& \mathbb{E}(-,A) \arrow[r, Rightarrow, "\mathbb{E}(-{,}x)"]
& \mathbb{E}(-,B) 
\end{tikzcd}
\end{center}
\end{enumerate}
\end{proposition}

\begin{lemma}\label{weakKernelCokernel}\cite[Lemma 3.2]{NakaokaPalu}. Let $(\mathscr{C},\mathbb{E},\mathfrak{s})$ be a triple satisying (ET1),(ET2), (ET3), (ET3)$^{\text{op}}$. Then for any $\mathbb{E}$-triangle, 
\begin{tikzcd}
A \arrow[r,"x"]
& B \arrow[r,"y"]
& C \arrow[r,dashed, "\d"]
& \text{,}
\end{tikzcd}
the following statements hold:
\begin{enumerate}
\item $y \circ x = 0,$
\item $x_{*}\d = 0,$
\item $y^{*}\d = 0.$
\end{enumerate}
\end{lemma}

\begin{proposition}\label{mappingCone}\cite[Proposition 1.20]{liu2019hearts}. Let $(\mathscr{C},\mathbb{E},\mathfrak{s})$ be an extriangulated category. Let 
\begin{center}
\begin{tikzcd} 
A \arrow[r, "x"]
& B \arrow[r,"y"]
& C \arrow[r, dashed, "\d"] &\text{}
\end{tikzcd}
\end{center}
be an $\mathbb{E}$-triangle, let $f \colon A \rightarrow D$ be any morphism and let 
\begin{center}
\begin{tikzcd} 
D \arrow[r, "d"]
& E \arrow[r,"e"]
& F \arrow[r, dashed, "f_{*}\d"] &\text{}
\end{tikzcd}
\end{center} 
be an $\mathbb{E}$-triangle realising $f_{*}\d$. Then there is a morphism $g$ such that the following diagram commutes 
\begin{center}
\begin{tikzcd} 
A \arrow[r, "x"] \arrow[d,"f"]
& B \arrow[r,"y"] \arrow[d,"g"]
& C \arrow[r, dashed, "\d"] \arrow[d,equal] &\text{,}
\\
D \arrow[r,"d"] 
& E \arrow[r,"e"] 
& C \arrow[r,dashed, "f_{*}\d"]
& \text{}
\end{tikzcd}
\end{center}
and that \begin{tikzcd} 
A \arrow[r, "\begin{pmatrix}
-f \\ x
\end{pmatrix}"]
& D \oplus B \arrow[r,"\begin{pmatrix}
d \hspace{0.3cm} g
\end{pmatrix}"]
& E \arrow[r, dashed, "e^{*}\d"] &\text{}
\end{tikzcd}
is an $\mathbb{E}$-triangle.

Dually, let 
\begin{center}
\begin{tikzcd} 
A \arrow[r, "x"]
& B \arrow[r,"y"]
& C \arrow[r, dashed, "\d"] &\text{}
\end{tikzcd}
\end{center}
be an $\mathbb{E}$-triangle, let $h \colon E \rightarrow C$ be any morphism and let 
\begin{center}
\begin{tikzcd} 
A \arrow[r, "d"]
& D \arrow[r,"e"]
& E \arrow[r, dashed, "h^{*}\d"] &\text{}
\end{tikzcd}
\end{center}
be an $\mathbb{E}$-triangle realising $h^{*}\d$. Then there is a morphism $g \colon D \rightarrow B$ such that the following diagram commutes
\begin{center}
\begin{tikzcd} 
A \arrow[r,"d"] \arrow[d,equal]
& D \arrow[r,"e"] \arrow[d,"g"]
& E \arrow[r,dashed, "h^{*}\d"] \arrow[d,"h"]
& \text{} 
\\
A \arrow[r, "x"] 
& B \arrow[r,"y"] 
& C \arrow[r, dashed, "\d"] 
&\text{}
\end{tikzcd}
\end{center}
and that \begin{tikzcd} 
D \arrow[r, "\begin{pmatrix}
-e \\ g
\end{pmatrix}"]
& E \oplus B \arrow[r,"\begin{pmatrix}
h \hspace{0.3cm} y
\end{pmatrix}"]
&C \arrow[r, dashed, "d_{*}\d"] &\text{}
\end{tikzcd} is an $\mathbb{E}$-triangle. 

\end{proposition}

The following proposition is a special case of \cite[Proposition 3.5]{herschend2017n} which applies to general $n$-exangulated categories. But here we are restating the statement just for extriangulated categories, which are the same as $1$-exangulated categories. The statement of \cite[Proposition 3.5]{herschend2017n} is a consequence of (ET2) and Proposition \ref{mappingCone}, or axioms (R0) and (EA2) respectively in the language of \cite{herschend2017n}. 

\begin{corollary}\label{mappingCone2}\cite[Proposition 3.5(2)]{herschend2017n}. Let $(\mathscr{C},\mathbb{E},\mathfrak{s})$ be an extriangulated category. Let \begin{tikzcd} 
A \arrow[r, "a"]
& B \arrow[r,"b"]
& C \arrow[r, dashed, "\varepsilon"] &\text{}
\end{tikzcd}
and 
\begin{tikzcd} 
X \arrow[r, "x"]
& Y \arrow[r,"y"]
& Z \arrow[r, dashed, "\d"] &\text{}
\end{tikzcd} be $\mathbb{E}$-triangles. Suppose we have the following solid commutative diagram.
\begin{center}
\begin{tikzcd} 
A \arrow[r, "a"] \arrow[d,equal]
& B \arrow[r,"b"] \arrow[d,"u"]
& C \arrow[r, dashed, "\varepsilon"] &\text{}
\\
A \arrow[r, "x"]
& Y \arrow[r,"y"]
& Z \arrow[r, dashed, "\d"] &\text{}
\end{tikzcd}
\end{center}
Then there exists a morphism $w \colon C \rightarrow Z$ such $wb = yu, w^{*}\d = \varepsilon$ and that the following is an $\mathbb{E}$-triangle,

\begin{center}
\begin{tikzcd} 
B \arrow[r, "\begin{pmatrix}
-b \\ u
\end{pmatrix}"]
& C \oplus Y \arrow[r,"\begin{pmatrix}
w \hspace{0.3cm} y
\end{pmatrix}"]
& Z \arrow[r, dashed, "a_{*}\d"] &\text{.}
\end{tikzcd}
\end{center}

\end{corollary}

\section{Idempotent completion of extriangulated categories.}
For the rest of this section, let $(\mathscr{C}, \mathbb{E}, \mathfrak{s})$ be an extriangulated category and let $\tilde{\mathscr{C}}$ and $\hat{\mathscr{C}}$ be the idempotent completion of $\mathscr{C}$ and the weak idempotent completion of $\mathscr{C}$ respectively. Note that in order to consider $\hat{\mathscr{C}}$, we have to further assume that $\mathscr{C}$ is a small category. 
\subsection{Idempotent completion.}
\begin{theorem} Let $(\mathscr{C}, \mathbb{E}, \mathfrak{s})$ be an extriangulated category. Let $\tilde{\mathscr{C}}$ be the idempotent completion of $\mathscr{C}$. Then $\tilde{\mathscr{C}}$ is extriangulated. Moreover, in this case the embedding $i_{\mathscr{C}} \colon \mathscr{C} \rightarrow \tilde{\mathscr{C}}$ is an extriangulated functor.
\end{theorem}

Our first step in proving the above theorem is the construction of a bifunctor $\mathbb{F} \colon \tilde{\mathscr{C}}^{\text{op}} \times \tilde{\mathscr{C}} \rightarrow Ab$ for the extriangulated structure. Given a pair of objects $(X,p)$ and $(Y,q)$ in $\tilde{\mathscr{C}}$, we define $\mathbb{F}$ on objects by setting, $$\mathbb{F}((X,p),(Y,q)) := p^{*}q_{*}\mathbb{E}(X,Y) = \{p^{*}q_{*}\d \mid \d \in \mathbb{E}(X,Y)\}.$$

\begin{lemma}\label{Extsubgrp} Let $p \colon X \rightarrow X$ and $q \colon Y \rightarrow Y$ be morphisms in $\mathscr{C}$. Then $p^{*}q_{*}\mathbb{E}(X,Y) = \{p^{*}q_{*}\d \mid \d \in \mathbb{E}(X,Y)\}$ is a subgroup of $\mathbb{E}(X,Y)$.
\begin{proof} 
Observe that $p^{*}q_{*}\mathbb{E}(X,Y)$ is the image of $\mathbb{E}(X,Y)$ under the group homomorphism $\mathbb{E}(p,q)$, therefore $p^{*}q_{*}\mathbb{E}(X,Y)$ is a subgroup of $\mathbb{E}(X,Y)$. 
\end{proof}
\end{lemma}
By Lemma \ref{Extsubgrp}, $p^{*}q_{*}\mathbb{E}(X,Y)$ is an abelian group. We now need to define $\mathbb{F}$ on morphisms. 

Let $\tilde{\alpha} \colon (X,p) \rightarrow (Y,q)$ and $\tilde{\beta} \colon (U,e) \rightarrow (V,f)$ be any pair of morphisms in $\tilde{\mathscr{C}}$.  By definition these are morphisms $\alpha \colon X \rightarrow Y$ and $\beta \colon U \rightarrow V$ in $\mathscr{C}$ such that $\alpha p = q \alpha = \alpha $ and $\beta e = f \beta = \beta.$ Take $\varepsilon \in \mathbb{F}((Y,q),(U,e))$, we have that $\varepsilon = q^{*}e_{*}\d_{\varepsilon}$ for some $\d_{\varepsilon} \in \mathbb{E}(Y,U)$, hence we observe that
$$ \beta_{*}\alpha^{*}\varepsilon = \beta_{*}\alpha^{*}q^{*}e_{*}\d_{\varepsilon} = \b_{*}e_{*}\a^{*}q^{*}\d_{\varepsilon} = (\b e)_{*}(q\a)^{*}\d_{\varepsilon} = (f\b)_{*}(\a p)^{*}\d_{\varepsilon} = f_{*}\b_{*}p^{*}\a^{*}\d_{\varepsilon} = p^{*}f_{*}(\a^{*}\b_{*}\d_{\varepsilon}).$$
Since $\a^{*}\b_{*}\d_{\varepsilon}$ is in $\mathbb{E}(X,V)$ we have that $\alpha^{*}\beta_{*}\varepsilon$ is an element of $\mathbb{F}((X,p),(V,f))$. 

For the pair $(\tilde{\alpha},\tilde{\beta})$ we define $\mathbb{F}(\tilde{\alpha}^{\text{op}}, \tilde{\beta}) \colon \mathbb{F}((Y,q),(U,e)) \rightarrow \mathbb{F}((X,p),(V,f))$ as follows. For $\varepsilon \in \mathbb{F}((Y,q),(U,e))$ we set $\mathbb{F}(\tilde{\alpha}^{\text{op}}, \tilde{\beta})(\varepsilon) := \beta_{*}\alpha^{*}\varepsilon.$ It is easy to observe that $\mathbb{F}$ preserves identity morphisms from the above definition. Let $(\tilde{\alpha_{1}}, \tilde{\beta_{1}})$ and $(\tilde{\alpha_{2}}, \tilde{\beta_{2}})$ be a pair of composable morphisms in $\tilde{\mathscr{C}}^{\text{op}} \times \tilde{\mathscr{C}}$ and $(\tilde{\alpha_{1}} \tilde{\alpha_{2}} , \tilde{\beta_{1}} \tilde{\beta_{2}} )$ be their composition. Then,
$$\mathbb{F}((\tilde{\alpha_{1}} \tilde{\alpha_{2}})^{\text{op}}, \tilde{\beta_{1}} \tilde{\beta_{2}})(\varepsilon) = \mathbb{F}(\tilde{\alpha_{2}}^{\text{op}} \tilde{\alpha_{1}}^{\text{op}}, \tilde{\beta_{1}} \tilde{\beta_{2}})(\varepsilon) = \b_{1*}\b_{2*}(\a_{2}\a_{1})^{*}\varepsilon=\b_{1*}\b_{2*}\a_{1}^{*}\a_{2}^{*}\varepsilon = \b_{1*}\a_{1}^{*}\b_{2*}\a_{2}^{*}\varepsilon,$$
so $\mathbb{F}$ preserves composition. This completes the definition of the bifunctor $\mathbb{F} \colon \tilde{\mathscr{C}}^{\text{op}} \times \tilde{\mathscr{C}} \rightarrow Ab$. 

Our next step will be to verify that $\mathbb{F} \colon \tilde{\mathscr{C}}^{\text{op}} \times \tilde{\mathscr{C}} \rightarrow Ab$ is a biadditive functor. We will only show that $\mathbb{F}$ is additive in the second argument since the proof for additivity in the first argument is dual. 

\begin{proposition}\label{biadd1}
Fix $(X,p)$ in $\tilde{\mathscr{C}}$. Then the functor $\mathbb{F}((X,p), -) \colon \tilde{\mathscr{C}} \rightarrow Ab$ is an additive functor. 
\begin{proof}
For the zero object $(0,1_0)$ in $\tilde{\mathscr{C}}$, we have that $\mathbb{F}((X,p),(0,1_0)) = p^{*}(1_{0})_{*}\mathbb{E}(X,0)=\{0\}.$

Now let $(U,e)$ and $(V,f)$ be any pair of objects in $\tilde{\mathscr{C}}$. Denote by $\mathbb{F}_{X}^{U \oplus V}$ the abelian group $\mathbb{F}((X,p),(U \oplus V, e \oplus f)) = \{ p^{*}(e \oplus f)_{*} \d \mid \d \in \mathbb{E}(X,U \oplus V)\}$. We also denote by $\mathbb{F}_{X}^{U}$ the abelian group $\mathbb{F}((X,p),(U,e)) = \{ p^{*}e_{*}\varepsilon \mid \varepsilon \in \mathbb{E}(X,U)\}$. We likewise denote by  $\mathbb{F}_{X}^{V}$ the abelian group $\mathbb{F}((X,p),(V,f)) = \{ p^{*}f_{*}\t \mid \t \in \mathbb{E}(X,V)\}$.

Since $\mathbb{E}$ is a biadditive functor, there is a group isomorphism $\varphi \colon \mathbb{E}(X,U \oplus V) \rightarrow \mathbb{E}(X,U) \oplus \mathbb{E}(X,V)$, where an $\mathbb{E}$-extension $\d \in \mathbb{E}(X,U \oplus V)$ corresponds to $\varphi( \d) = (\d_{U},\d_{V})$ for some $\d_{U} \in \mathbb{E}(X,U)$ and $\d_{V} \in \mathbb{E}(X,V)$. 

Define the map $G \colon \mathbb{F}_{X}^{U \oplus V} \rightarrow \mathbb{F}_{X}^{U} \oplus \mathbb{F}_{X}^{V}$ by setting $G(p^{*}(e \oplus f)_{*}\d) = (p^{*}e_{*}\d_{U}, p^{*}f_{*}\d_{V})$ where $\varphi(\d) =(\d_{U},\d_{V})$ for some $\d_{U} \in \mathbb{E}(X,U)$ and $\d_{V} \in \mathbb{E}(X,V)$. Observe that for any pair $p^{*}(e \oplus f)_{*}\d$ and $ p^{*}(e \oplus f)_{*}\varepsilon$ where $\varphi(\d)=(\d_{U},\d_{V})$ and $\varphi(\varepsilon)=(\varepsilon_{U}, \varepsilon_{V})$ we have that $\varphi(\d + \varepsilon) = \varphi(\d) + \varphi(\varepsilon).$ So 
$$G(p^{*}(e \oplus f)_{*}\d + p^{*}(e \oplus f)_{*}\varepsilon) = G(p^{*}(e \oplus f)_{*}(\d + \varepsilon))=(p^{*}e_{*}(\d_U + \varepsilon_U), p^{*}e_{*}(\d_V +  \varepsilon_V))$$
$$= (p^{*}e_{*}\d_U, p^{*}f_{*}\d_V) + (p^{*}e_{*}\varepsilon_U, p^{*}f_{*}\varepsilon_V) = G(p^{*}(e \oplus f)_{*}\d) + G(p^{*}(e \oplus f)_{*}\varepsilon).$$
Hence $G$ is a group homomorphism. 

Define the map $H \colon \mathbb{F}_{X}^{U} \oplus \mathbb{F}_{X}^{V} \rightarrow \mathbb{F}_{X}^{U \oplus V}$ by setting $H(p^{*}e_{*}\d_U,p^{*}f_{*}\d_{V}) = p^{*}(e \oplus f)_{*}\d$ where $\varphi^{-1}(\d_U,\d_V) = \d$ for some $\d \in \mathbb{E}(X,U \oplus V)$, and where $\d_{U} \in \mathbb{E}(X,U)$ and $\d_{V} \in \mathbb{E}(X,V)$. Take any pair $(p^{*}e_{*}\d_{U} , p^{*}f_{*}\d_{V})$ and $(p^{*}e_{*}\varepsilon,p^{*}f_{*}\varepsilon_{V})$ in $\mathbb{F}_{X}^{U} \oplus \mathbb{F}_{X}^{V}$. Then if $\varphi^{-1}(\d_U,\d_V) = \d$ and $\varphi^{-1}(\varepsilon_U,\varepsilon_V) = \varepsilon$ then $\varphi^{-1}((\d_{U} + \varepsilon_{U}, \d_{V} + \varepsilon_{V})) =  \varphi^{-1}(\d_U,\d_V)+ \varphi^{-1}(\varepsilon_U,\varepsilon_V)$. So
$$ H((p^{*}e_{*}\d_U,p^{*}f_{*}\d_{V}) +(p^{*}e_{*}\varepsilon_{U},p^{*}f_{*}\varepsilon_{V}) ) = H((p^{*}e_{*}(\d_U + \varepsilon_U),p^{*}f_{*}(\d_{V} + \varepsilon_{V}))$$ 
$$= p^{*}(e \oplus f)_{*}(\d + \varepsilon) =p^{*}(e \oplus f)_{*} \d + p^{*}(e \oplus f)_{*} \varepsilon = H(p^{*}e_{*}\d_U,p^{*}f_{*}\d_{V}) + H(p^{*}e_{*}\varepsilon_{U},p^{*}f_{*}\varepsilon_{V}).$$ Hence $H$ is a group homomorphism.

We claim that $G \circ H = 1_{\mathbb{F}_{X}^{U} \oplus \mathbb{F}_{X}^{V}}$. Take $(p^{*}e_{*}\d_U,p^{*}f_{*}\d_V) 
\in \mathbb{F}_{X}^{U} \oplus \mathbb{F}_{X}^{V}$ and suppose $\varphi^{-1}(\d_U, \d_V) = \d$. Then $H((p^{*}e_{*}\d_U,p^{*}f_{*}\d_V)) = p^{*}(e \oplus f)_{*}\d$. Since $\phi(\d) =(\d_U,\d_V)$, we have that 
$$GH((p^{*}e_{*}\d_U,p^{*}f_{*}\d_V))=G(p^{*}(e \oplus f)_{*}\d)=(p^{*}e_{*}\d_U,p^{*}f_{*}\d_V).$$

We also claim that $H \circ G = 1_{\mathbb{F}_{X}^{U \oplus V}}$. Take $p^{*}(e \oplus f)_{*}\d \in \mathbb{F}_{X}^{U \oplus V}$ and suppose $\varphi(\d) = (\d_U, \d_V)$. Then $G(p^{*}(e \oplus f)_{*}\d) = (p^{*}e_{*}\d_U, p^{*}f_{*}\d_V)$. Since $\varphi^{-1}(\d_U,\d_V) = \d$, we have that 
$$ H(G(p^{*}(e \oplus f)_{*}\d))=H((p^{*}e_{*}\d_U, p^{*}f_{*}\d_V))=p^{*}(e \oplus f)_{*}\d.$$
This shows that $\mathbb{F}((X,p),(U \oplus V, e \oplus f)) \cong \mathbb{F}((X,p),(U,e)) \oplus \mathbb{F}((X,p),(V,f))$. Therefore the functor $\mathbb{F}((X,p), -) \colon \tilde{\mathscr{C}} \rightarrow Ab$ is additive. 
\end{proof}
\end{proposition}

\begin{proposition}\label{biadd2}
Fix $(X,p)$ in $\tilde{\mathscr{C}}$. Then the functor $\mathbb{F}(-,(X,p)) \colon \tilde{\mathscr{C}}^{\text{op}} \rightarrow Ab$ is an additive functor. 
\begin{proof}
The proof is dual to the proof of the previous proposition. 
\end{proof}
\end{proposition}

Having verified that the functor $\mathbb{F}$ is biadditive, the next thing we need to do is to define a correspondence which will be a realisation. In order to define the correspondence, we need the following lemma. This lemma is a generalisation of \cite[Lemma 1.13]{BalSch} in the setting of extriangulated categories. The proof is also a straightforward adaptation. 

\begin{lemma}\label{IdemFillingMorph} Let $(\mathscr{A}, \mathbb{G},\mathfrak{t})$ be a triple satisfying (ET1) and (ET2). Let $A,B,C$ be objects of $\mathscr{A}$. Let $\d$ be an extension in $\mathbb{G}(C,A)$ with $\mathfrak{t}(\d) = [ A \overset{a}{\longrightarrow} B \overset{b}{\longrightarrow} C]$. Let $(e,f) \colon \d \rightarrow \d$  be a morphism of $\mathbb{G}$-extensions where $e \colon A \rightarrow A$ and $f \colon C \rightarrow C$ are idempotent morphisms.  Then there exists an idempotent morphism $g \colon B \rightarrow B$ such that the triple $(e,g,f)$ realises the morphism of $\mathbb{G}$-extensions $(e,f) \colon \d \rightarrow \d$. 
\begin{center}
\begin{tikzcd}
A \arrow[r,"a"] \arrow[d,"e"]
& B \arrow[r,"b"] \arrow[d, dashed, "g"]
& C \arrow[d,"f"]
\\
A \arrow[r,"a"]
& B \arrow [r,"b"]
& C
\end{tikzcd}
\end{center}

\begin{proof}
Since $(e,f) \colon \d \rightarrow \d$ is a morphism of $\mathbb{G}$-extensions and $\mathfrak{t}$ is a realisation, there exists a morphism $i \colon B \rightarrow B$ such that the following diagram commutes. 
\begin{center}
\begin{tikzcd}
A \arrow[r,"a"] \arrow[d,"e"]
& B \arrow[r,"b"] \arrow[d, dashed, "i"]
& C \arrow[d,"f"]
\\
A \arrow[r,"a"]
& B \arrow [r,"b"]
& C
\end{tikzcd}
\end{center}
Let $h := i^{2}-i$. Then we have that $ha = (i^2-i)a = 0$ and $bh=b(i^{2}-i)=0$ from the commutativity of the above diagram and the fact that $e$ and $f$ are idempotent. By the exact sequences in Proposition \ref{longExact}, $b$ is a weak cokernel of $a$ so there exists $\bar{h} \colon C \rightarrow B$ such that $h=\bar{h}b.$ So we observe that $h^{2}=\bar{h}bh =0$. 

Let $g = i + h -2ih$. Since the morphisms $i$ and $h$ commute and $h^2=0$ we have that $g^{2} = i^2 + 2ih-4i^2h$. Since $i^2 = i + h$ we have that $g^2 = i +h + 2ih -4ih =g$. We have that $ga = ia + ha -2iha = ia = ae$, since $ha=0$. We likewise have that $bg = bi + bh-2bih = bi + bh -2bhi = bi = fb$. Therefore, the above diagram commutes if we replace $i$ with $g$. This completes the proof. 
\end{proof}
\end{lemma}

\begin{lemma}\label{IdemFillingMorph2}
Let $(\mathscr{A}, \mathbb{G},\mathfrak{t})$ be a triple satisfying (ET1) and (ET2). Let $A,B,C$ be objects of $\mathscr{A}$. Let $\d$ be an extension in $\mathbb{G}(C,A)$ with $\mathfrak{t}(\d) = [ A \overset{a}{\longrightarrow} B \overset{b}{\longrightarrow} C]$. Let $(e,f) \colon \d \rightarrow \d$  be a morphism of $\mathbb{G}$-extensions realised by $(e,i,f)$ where $e \colon A \rightarrow A$ and $i \colon B \rightarrow B$ are idempotent morphisms. Then there exists an idempotent morphism $g \colon C \rightarrow C$ such that $(e,g) \colon \d \rightarrow \d$ is a morphism of $\mathbb{G}$-extensions realised by $(e,i,g)$. 
Dually if we instead assume that $i \colon B \rightarrow B$ and $f \colon C \rightarrow C$ are idempotent. Then there exists an idempotent morphism $k \colon A \rightarrow A$ such that $(k,f) \colon \d \rightarrow \d$ is a morphism of $\mathbb{G}$-extensions realised by $(k,i,f)$. 
\begin{proof}
Since $(e,f) \colon \d \rightarrow \d$ is realised by $(e,i,f)$, we have the following commutative diagram. 
\begin{center}
\begin{tikzcd}
A \arrow[r,"a"] \arrow[d,"e"]
& B \arrow[r,"b"] \arrow[d, "i"]
& C \arrow[d,"f"]
\\
A \arrow[r,"a"]
& B \arrow [r,"b"]
& C
\end{tikzcd}
\end{center}
Set $h := f^2 - f$. Then we have that $$hb = f^{2}b - fb = f(bi)-bi = (fb)i-bi=bi^{2}-bi = bi - bi =0.$$ We also have that
$$h^{*}\d = (f^{2}-f)^{*}\d = f^{*}f^{*}\d - f^{*}\d = f^{*}e_{*}\d - e_{*}\d = e_{*}(f^{*}\d)-e_{*}\d = (e^{2})_{*}\d -e_{*}\d=e_{*}\d - e_{*}\d=0.$$ 
By Proposition \ref{longExact} we have the following exact sequence in $Ab$.
\begin{center}
\begin{tikzcd}
\mathscr{C}(C,A) \arrow[r, "\mathscr{C}(C{,}a)"]
& \mathscr{C}(C,B) \arrow[r, "\mathscr{C}(C{,}b)"]
& \mathscr{C}(C,C) \arrow[r, "(\d_{\#})_{C}"]
& \mathbb{E}(C,A) \arrow[r, "\mathbb{E}(C{,}a)"]
& \mathbb{E}(C,B) \arrow[r,"\mathbb{E}(C{,}b)"]
& \mathbb{E}(C,C)
\end{tikzcd}
\end{center}
Since $h^{*}\d = (\d_{\#})_{C}(h) =0$, it follows from the exactness of the above sequence that there exists a morphism $\bar{h} \colon C \rightarrow B$ such that $h = b\bar{h}$. From this we can observe that $h^{2} = b\bar{h}b\bar{h} = hb\bar{h} = 0$. Now set $g:= f+h-2fh$, as $f$ and $h$ commute and $h^2=0$, we then have that
$$g^{2} = (f+h-2fh)^2 = f^2 + 2fh-4f^2h.$$
By noting that $f^2 = f +h$ we obtain 
$$g^2 = f + h -2fh = g.$$
It is then easy to check that,
$$g^{*}\d = f^{*}\d +h^{*}\d -2h^{*}f^{*}\d = e_{*}\d + 0 -2h^{*}e_{*}\d = e_{*}\d -2e_{*}h^{*}\d = e_{*}\d -2e_{*}0 = e_{*}\d,$$
and $$gb = (f+h-2fh)b= fb + hb -2fhb = bi + 0 -2f0 =bi.$$
We have shown that $g \colon C \rightarrow C$ is an idempotent morphism such that $(e,g) \colon \d \rightarrow \d$ is a morphism of $\mathbb{G}$-extensions realised by $(e,i,g)$.  The proof of the other statement is dual.
\end{proof}
\end{lemma}

\begin{definition} Let $\mathfrak{r}$ be the correspondence between $\mathbb{F}$-extensions and equivalence classes of sequences of morphisms in $\tilde{\mathscr{C}}$ defined as follows. For any objects $Z,X$ in $\mathscr{C}$ and idempotent morphisms $p \colon Z \rightarrow Z, q \colon X \rightarrow X$ in $\mathscr{C}$, let $\d = p^{*}q_{*}\varepsilon$ be an extension in $\mathbb{F}((Z,p),(X,q))$ such that $$\mathfrak{s}(p^{*}q_{*}\varepsilon) = [ X \overset{x}{\longrightarrow} Y \overset{y}{\longrightarrow} Z].$$ 

We set $$\mathfrak{r}(\d) := [ (X,q) \overset{xq}{\longrightarrow} (Y,r) \overset{py}{\longrightarrow} (Z,p)],$$ where $r \colon Y \rightarrow Y$ is an idempotent morphism such that $rx=xq$ and $yr=py$ obtained by application of Lemma \ref{IdemFillingMorph}. 
\end{definition}

\begin{remark}\label{wellDefn} Before we can proceed any further, we need to show that the above definition of $\mathfrak{r}$ is well-defined in the following sense. Given an $\mathbb{F}$-extension $\d$, $\mathfrak{r}(\d)$ is defined in terms of a choice of the representative $\mathfrak{s}(\d)$. We will show that it is independent of this choice. Moreover, in the above definition, the idempotent morphism $r \colon Y \rightarrow Y$ such that $rx=xq$ and $yr = py$ need not be unique. We will show that all choices of such an idempotent give equivalent sequences.
\end{remark}

\begin{lemma}\label{shortExactseq1} Let $\d = p^{*}q_{*}\varepsilon$ be an extension in $\mathbb{F}((Z,p),(X,q))$ such that $$\mathfrak{s}(p^{*}q_{*}\varepsilon) = [ X \overset{x}{\longrightarrow} Y \overset{y}{\longrightarrow} Z].$$ For any object $(A,e) \in \tilde{\mathscr{C}}$, the following sequence in $Ab$ is exact.
\begin{center}
\begin{tikzcd}[column sep=4.5em,row sep=4.5em]
\tilde{\mathscr{C}}((Z,p),(A,e)) \arrow[r, "\tilde{\mathscr{C}}(py{,}(A{,}e))"]
& \tilde{\mathscr{C}}((Y,r),(A,e))  \arrow[r, "\tilde{\mathscr{C}}(xq{,}(A{,}e))"]
& \tilde{\mathscr{C}}((X,q),(A,e)) 
\end{tikzcd}
\end{center}
In particular, in the sequence 
\begin{center}
\begin{tikzcd}
(X,q) \arrow[r,"xq"] 
& (Y,r) \arrow[r,"py"]
& (Z,p) 
\end{tikzcd}
\end{center}
the morphism $py$ is a weak cokernel of $xq$ and the morphism $xq$ is a weak kernel of $py$.

\begin{proof}
Since $(\mathscr{C},\mathbb{E},\mathfrak{s})$ is extriangulated and $\mathfrak{s}(p^{*}q_{*}\varepsilon) = [ X \overset{x}{\longrightarrow} Y \overset{y}{\longrightarrow} Z]$, we have that $y \circ x = 0$ by Lemma \ref{weakKernelCokernel}. So it follows that for $f \in \tilde{\mathscr{C}}((Z,p),(A,e))$, $$(f  \circ py) \circ (xq) = (fp)(y\circ x) = 0.$$ That is to say im($\tilde{\mathscr{C}}(py{,}(A{,}e))) \subseteq \text{ker}(\tilde{\mathscr{C}}(xq{,}(A{,}e)))$.

Recall that $r$ is an idempotent such that $xq = rx$ and $py = yr$. Given $g \in \tilde{\mathscr{C}}((Y,r),(A,e))$ such that $g \circ xq = 0$, we have that $gxq = grx =0$. By the exactness of the sequence in Proposition \ref{longExact}, $y$ is a weak cokernel of $x$, so we have that there exists a morphism $h \colon Z \rightarrow A$ such that $gr = hy$. Since $gr =hy$ we have that $$gr = gr^2 = hyr = hpy.$$ Moreover for the morphism $ehp \colon (Z,p) \rightarrow (A,e)$ we have that $$gr = egy = ehpy = ehppy = (ehp) \circ py.$$ This is to say $gr \in \text{im}(\tilde{\mathscr{C}}(py{,}(A{,}e)))$, in particular $\text{ker}(\tilde{\mathscr{C}}(xq{,}(A{,}e))) \subseteq \text{ im}\tilde{\mathscr{C}}(py{,}(A{,}e))).$ This completes the proof.
\end{proof}
\end{lemma}

\begin{proposition}\label{independentOfRep} Let $\d$ be an extension in $\mathbb{F}((C,p),(A,q))$ realised under $\mathfrak{s}$ by the following sequences,  $$A \overset{a}{\longrightarrow} B \overset{b}{\longrightarrow} C,$$
$$A \overset{x}{\longrightarrow} Y \overset{y}{\longrightarrow} C.$$ Then given idempotents $r \colon B \rightarrow B$ and $w \colon Y \rightarrow Y$ such that 
\begin{equation}\label{idemRelns1} aq = ra, \text{ } pb = br  \text{ and } xq = wx, \text{ } py=yw \end{equation} the following sequences are equivalent,
$$(A,q) \overset{aq}{\longrightarrow} (B,r) \overset{pb}{\longrightarrow} (C,p),$$
$$(A,q) \overset{xq}{\longrightarrow} (Y,w) \overset{py}{\longrightarrow} (C,p).$$
That is to say, $\mathfrak{r}$ is well-defined.
\begin{proof}
Since the sequences $A \overset{a}{\longrightarrow} B \overset{b}{\longrightarrow} C,$ and $A \overset{x}{\longrightarrow} Y \overset{y}{\longrightarrow} C$ both realise $\d$, they are by definition equivalent in $\mathscr{C}$. That is to say we have the following commutative diagram,
\begin{equation}\label{indepDiag1}
\begin{tikzcd}
A \arrow[r,"a"] \arrow[d,equal]
& B \arrow[r,"b"] \arrow[d,"f"]
& C \arrow[d,equal]
\\
A \arrow[r,"x"]
&Y \arrow[r,"y"]
&C 
\end{tikzcd}
\end{equation}
where the morphism $f \colon B \rightarrow Y$ is an isomorphism. Now consider the following diagram. 

\begin{equation}\label{indepDiag2}
\begin{tikzcd}
(A,q) \arrow[r,"aq"] \arrow[d,equal]
& (B,r) \arrow[r,"pb"] \arrow[d,"wfr"]
& (C,p) \arrow[d,equal]
\\
(A,q) \arrow[r,"xq"]
& (Y,w) \arrow[r,"py"]
& (C,p) 
\end{tikzcd}
\end{equation}
From the relations in (\ref{idemRelns1}) and those arising from the commutative diagram (\ref{indepDiag1}), we can observe the following,
\begin{equation}\label{eqnSet1}  wf(raq) = wf(aq) = wxq = xq, \end{equation}
\begin{equation}\label{eqnSet2} py(wfr) = p(yf)r = p(br) = pb. \end{equation}
That is to say, diagram (\ref{indepDiag2}) commutes. 

From (\ref{eqnSet1}) we can see that $wfraq = wfaq$ therefore $(wfr-wf)aq =0$. By Lemma \ref{shortExactseq1}, $pb$ is a weak cokernel of $aq$, so there exists a morphism $h \colon (C,p) \rightarrow (Y,w)$ such that $wfr-wf = hpb$. Since $r$ is idempotent and satisfies the relations in (\ref{idemRelns1}), we have that $$0 = wfr -wfr = wfr^2 -wfr = hpbr = hpb,$$ in particular $$wfr -wf = hpb = 0.$$ Hence $wfr=wf$. 

From (\ref{eqnSet2}) we have that $py(wfr) = py(fr)$ therefore $py(wfr-fr) = 0.$ By Lemma \ref{shortExactseq1}, we have that $xq$ is a weak kernel of $py$ and there exists a morphism $g \colon (B,r) \rightarrow (A,q)$ such that $wfr-fr = xqg.$ Since $w$ is idempotent and satisfies the relations in (\ref{idemRelns1}), we have that $$0 = wfr -wfr = w^2fr -wfr =wxqg =xqg,$$ in particular $$wfr -fr = xqg = 0.$$ Hence $wfr=fr$.

Now consider the morphism $rf^{-1}w \colon (Y,w) \rightarrow (B,r)$ in $\tilde{\mathscr{C}}$. Since $wfr = wf$ and $wfr=fr$, we can observe that
$$ wfr \circ rf^{-1}w = wfrrf^{-1}w = (wfr)f^{-1}w = (wf)f^{-1}w = w,$$
$$rf^{-1}w \circ wfr =rf^{-1}wwfr =rf^{-1}(wfr) = rf^{-1}(fr) = r.$$ 
This is to say the morphisms $rf^{-1}w \colon (Y,w) \rightarrow (B,r)$ and $wfr \colon (B,w) \rightarrow (Y,w)$ are mutual inverses in $\tilde{\mathscr{C}}$. So diagram (\ref{indepDiag2}) does indeed give an equivalence of sequences in $\mathscr{C}$. This completes the proof.
\end{proof}
\end{proposition}
From Proposition \ref{independentOfRep} we conclude that $\mathfrak{r}$ is well-defined in the sense of Remark \ref{wellDefn}.
 
\begin{lemma}\label{standForm} Let $\d$ be an extension in $\mathbb{F}((Z,p),(X,q))$ with $\mathfrak{s}(\d)=[ X \overset{x}{\longrightarrow} Y \overset{y}{\longrightarrow} Z]$ and $\mathfrak{r}(\d) = [(X,q) \overset{xq}{\longrightarrow} (Y,r) \overset{py}{\longrightarrow} (Z,p)]$. Suppose that $(X,q) \overset{u}{\longrightarrow} (W,s) \overset{v}{\longrightarrow} (Z,p)$ is another sequence realising $\d$ as an $\mathbb{F}$-extension. Then $u=u_{1}q$ and $v=pv_{1}$ for some $u_{1} \colon X \rightarrow W$ and $v_{1} \colon W \rightarrow Z.$
\begin{proof} Since $(X,q) \overset{u}{\longrightarrow} (W,s) \overset{v}{\longrightarrow} (Z,p)$ realises $\d$, there is an equivalence, 
\begin{center}
\begin{tikzcd}
(X,q) \arrow[r,"xq"] \arrow[d,equal]
& (Y,r) \arrow[r,"py"] \arrow[d,"f"]
& (Z,p) \arrow[d,equal]
\\
(X,q) \arrow[r,"u"]
& (W,s) \arrow[r,"v"]
& (Z,p)
\end{tikzcd}
\end{center}
where $f$ is an isomorphism. Since the above diagram commutes $fxq = u$ and $vf = py$, so set $u_{1} = fx$ and $v_{1} = yf^{-1}$. 
\end{proof}
\end{lemma}

\begin{proposition}\label{addreal} Let $\mathfrak{r}$ be the correspondence defined above. Then $\mathfrak{r}$ is an additive realisation of $\mathbb{F}.$
\begin{proof}
 Let $\d = p^{*}q_{*}\varepsilon \in \mathbb{F}((Z,p),(X,q))$ and $\d^{\prime} = p^{\prime *}q^{\prime}_{*}\varepsilon^{\prime} \in \mathbb{F}((Z^{\prime},p^{\prime}),(X^{\prime},q^{\prime}))$ be $\mathbb{F}$-extensions with 
 $$\mathfrak{s}(\d) = [X \overset{x}{\longrightarrow} Y \overset{y}{\longrightarrow} Z], $$
 $$\mathfrak{r}(\d) = [(X,q) \overset{xq}{\longrightarrow} (Y,r) \overset{py}{\longrightarrow} (Z,p)]$$ 
 and 
 $$\mathfrak{s}(\d^{\prime}) = [X^{\prime} \overset{x^{\prime}}{\longrightarrow} Y^{\prime} \overset{y^{\prime}}{\longrightarrow} Z^{\prime}],$$
 $$\mathfrak{r}(\d^{\prime}) = [(X^{\prime},q^{\prime}) \overset{x^{\prime}q^{\prime}}{\longrightarrow} (Y^{\prime},r^{\prime}) \overset{p^{\prime}y^{\prime}}{\longrightarrow} (Z^{\prime},p^{\prime})].$$ 
 Suppose that we have a morphism of $\mathbb{F}$-extensions $(a,c) \colon \d \rightarrow \d^{\prime}$ for some $a \in \tilde{\mathscr{C}}((X,q),(X^{\prime},q^{\prime})$ and $c \in \tilde{\mathscr{C}}((Z,p),(Z^{\prime},p^{\prime})$, that is to say $\mathbb{F}((Z,p),a)(\d) = \mathbb{F}(c,(X,q))(\d^{\prime})$. In other words, we have the following diagram in $\tilde{\mathscr{C}}$.

\begin{equation}\label{RealIdem}
\begin{tikzcd}[column sep=3.0em,row sep=3.0em]
(X,q) \arrow[r, "xq"] \arrow[d,"a"]
& (Y,r) \arrow[r,"py"]
& (Z,p) \arrow[d,"c"] \\
(X^{\prime},q^{\prime}) \arrow[r, "x^{\prime}q^{\prime}"]
& (Y^{\prime},r^{\prime}) \arrow[r,"p^{\prime}y^{\prime}"]
& (Z^{\prime},p^{\prime}) 
\end{tikzcd}
\end{equation}

 By definition we have that $\d = p^{*}q_{*}\varepsilon$ and $\d^{\prime} = p^{\prime *}q^{\prime}_{*}\varepsilon^{\prime}$ for some $\varepsilon \in \mathbb{E}(Z,X)$ and some $\varepsilon^{\prime} \in \mathbb{E}(Z^{\prime},X^{\prime})$. Moreover the morphism $a \in \mathscr{C}(X,X^{\prime})$ is such that $aq = q^{\prime}a = a$, likewise for $ c \in \mathscr{C}(Z,Z^{\prime})$ we have that $cp = p^{\prime}c=c$. We also have by definition that $$\mathbb{F}((Z,p),a)(\d) = a_{*}(p^{*}q_{*}\varepsilon) = \mathbb{F}(c,(X,q))(\d^{\prime}) = c^{*}(p^{\prime *}q^{\prime}_{*}\varepsilon^{\prime}).$$
Therefore we have a morphism of $\mathbb{E}$-extensions $(a,c) \colon p^{*}q_{*}\varepsilon \rightarrow p^{\prime *}q^{\prime}_{*}\varepsilon^{\prime}.$ In other words we have the following solid diagram in $\mathscr{C}$.

\begin{center}
\begin{tikzcd}[column sep=3.5em,row sep=3.5em]
X \arrow[r, "x"] \arrow[d,"a"]
& Y \arrow[r,"y"] \arrow[d, dashed, "b"]
& Z \arrow[r, dashed, "p^{*}q_{*}\varepsilon"] \arrow[d,"c"] & \text{} \\
X \arrow[r, "x^{\prime}"]
& Y \arrow[r,"y^{\prime}"]
& Z \arrow[r, dashed, "p^{\prime *}q^{\prime}_{*}\varepsilon^{\prime}"] &\text{}
\end{tikzcd}
\end{center}

Since $\mathfrak{s}$ is a realisation, there exists a morphism $b \colon Y \rightarrow Y^{\prime}$ making the above diagram commute. Recall that by Lemma \ref{IdemFillingMorph}, we have that $rx=xq$, $yr=py,$ $r^{\prime}x^{\prime}=x^{\prime}q^{\prime}$ and $y^{\prime}r^{\prime}=p^{\prime}y^{\prime}$. It then follows that $r^{\prime}br \colon (Y,r) \rightarrow (Y^{\prime},r^{\prime})$ makes diagram (\ref{RealIdem}) commute since, 
$$r^{\prime}brxq =r^{\prime}b(rxq) = r^{\prime}b(xq^2) = r^{\prime}(bx)q = r^{\prime}x^{\prime}(aq)= (r^{\prime}x^{\prime})a = x^{\prime}q^{\prime}a$$ and 
$$p^{\prime}y^{\prime}r^{\prime}br = p^{\prime}(y^{\prime}r^{\prime})br = (p^{\prime} p^{\prime}) y^{\prime} br = p^{\prime} ( y^{\prime} b ) r =  p^{\prime}c(yr)= (p^{\prime}c)py = cpy.$$
 So we conclude that $\mathfrak{r}$ is a realisation of $\mathbb{F}$. 

Now we verify additivity of $\mathfrak{r}$. For any pair $(Z,p), (X,q)$, we have that $0 =p^{*}q_{*}0$ and $$\mathfrak{s}(0) = [X \overset{\big[\begin{smallmatrix}
1_X\\
0
\end{smallmatrix}\big]}{\longrightarrow} X \oplus Z \overset{[\begin{smallmatrix} 0 & 1_Z \end{smallmatrix}]}{\longrightarrow} Z]=0.$$ By definition we have that $$\mathfrak{r}(0) = [(X,q) \overset{\big[\begin{smallmatrix}
q\\
0
\end{smallmatrix}\big]}{\longrightarrow} (X,q) \oplus (Z,p) \overset{[\begin{smallmatrix} 0 & p \end{smallmatrix}]}{\longrightarrow} (Z,p)],$$
since $q = 1_{(X,q)}$ and $p = 1_{(Z,p)}$, we have that, 
$$ \mathfrak{r}(0) = 0.$$

Now take a pair of $\mathbb{F}$-extensions $\d = p^{*}q_{*}\varepsilon \in \mathbb{F}((Z,p),(X,q))$ and $\d^{\prime} = p^{\prime*}q^{\prime}_{*}\varepsilon^{\prime} \in \mathbb{F}((Z^{\prime},p^{\prime}),(X^{\prime},q^{\prime}))$. Since $\mathfrak{s}$ is an additive realisation we have that $$ \mathfrak{s}(p^{*}q_{*}\varepsilon \oplus p^{\prime*}q^{\prime}_{*}\varepsilon^{\prime}) = \mathfrak{s}(p^{*}q_{*}\varepsilon) \oplus \mathfrak{s}(p^{\prime*}q^{\prime}_{*}\varepsilon^{\prime}).$$
As $\mathfrak{s}(p^{*}q_{*}\varepsilon) =  [X \overset{x}{\longrightarrow} Y \overset{y}{\longrightarrow} Z]$ and $\mathfrak{s}(p^{\prime*}q^{\prime}_{*}\varepsilon^{\prime}) =  [X^{\prime} \overset{x^{\prime} }{\longrightarrow} Y^{\prime}  \overset{y^{\prime} }{\longrightarrow} Z^{\prime} ]$, we have that 
$$\mathfrak{s}(p^{*}q_{*}\varepsilon \oplus p^{\prime *}q^{\prime}_{*} \varepsilon^{\prime} ) = [X \oplus X^{\prime} \overset{x \oplus x^{\prime}}{\longrightarrow} Y \oplus Y^{\prime}  \overset{y \oplus y^{\prime}}{\longrightarrow}  Z \oplus Z^{\prime}]$$

By the definition of $\mathfrak{r}$ we have that,
$$\mathfrak{r}(\d) =  [(X,q) \overset{xq}{\longrightarrow} (Y,r) \overset{py}{\longrightarrow} (Z,p)] \text{, }$$

 $$\mathfrak{r}(\d^{\prime}) =  [(X^{\prime},q^{\prime}) \overset{x^{\prime}q^{\prime}}{\longrightarrow} (Y^{\prime},r^{\prime})  \overset{p^{\prime}y^{\prime} }{\longrightarrow} (Z^{\prime},p^{\prime})] \text{ and }$$
 
$$\mathfrak{r}(\d \oplus \d^{\prime}) = \begin{tikzcd}[column sep=4.0em,row sep=4.0em]
[(X \oplus X^{\prime}, q \oplus q^{\prime}) \arrow[r,"(x \oplus x^{\prime})(q \oplus q^{\prime})"]
& (Y \oplus Y^{\prime},r \oplus r^{\prime}) \arrow[r, "(p \oplus p^{\prime})(y \oplus y^{\prime})"]
& (Z \oplus Z^{\prime}, p \oplus p^{\prime})] \end{tikzcd}.$$
We have that $$(x \oplus x^{\prime})(q \oplus q^{\prime}) = \begin{pmatrix}
x & 0\\
0 & x^{\prime}
\end{pmatrix} \begin{pmatrix}
q & 0\\
0 & q^{\prime}
\end{pmatrix} = \begin{pmatrix}
xq & 0\\
0 & x^{\prime}q^{\prime}
\end{pmatrix}=(xq) \oplus (x^{\prime}q^{\prime}),$$ likewise $(p \oplus p^{\prime})(y \oplus y^{\prime}) = py \oplus p^{\prime}y^{\prime}$. It is also easy to check that $r \oplus r^{\prime}$ is idempotent and satisfies the required equations arising from Lemma \ref{IdemFillingMorph}. So it follows that $$\mathfrak{r}(\d \oplus \d^{\prime}) = \mathfrak{r}(\d) \oplus \mathfrak{r}(\d^{\prime}).$$ This completes the proof. 
\end{proof}
\end{proposition}
So far we have constructed the triple $(\tilde{\mathscr{C}}, \mathbb{F}, \mathfrak{r)}$. Since $\tilde{\mathscr{C}}$ is the idempotent completion of $\mathscr{C}$, it is an additive category. Propositions \ref{biadd1}, \ref{biadd2} and \ref{addreal} show that the triple $(\tilde{\mathscr{C}}, \mathbb{F}, \mathfrak{r)}$ satisfies axioms (ET1) and (ET2) of the definition of an extriangulated category, see Definition \ref{DefExtriang}. So what is left is to verify axioms (ET3), $\text{(ET3)}^{\text{op}}$, (ET4) and $\text{(ET4)}^{\text{op}}$.

\begin{proposition}\label{propet3} The triple $(\tilde{\mathscr{C}}, \mathbb{F}, \mathfrak{r)}$ satisfies the axioms (ET3) and $\text{(ET3)}^{\text{op}}$.
\begin{proof}
 Let $\d = p^{*}q_{*}\varepsilon \in \mathbb{F}((Z,p),(X,q))$ and $\d^{\prime} = (p^{\prime})^{*}(q^{\prime})_{*}\varepsilon^{\prime} \in \mathbb{F}((Z^{\prime},p^{\prime}),(X^{\prime},q^{\prime}))$ be $\mathbb{F}$-extensions with 
 $$\mathfrak{s}(\d) = [X \overset{x}{\longrightarrow} Y \overset{y}{\longrightarrow} Z],$$
 $$\mathfrak{r}(\d) = [(X,q) \overset{xq}{\longrightarrow} (Y,r) \overset{py}{\longrightarrow} (Z,p)],$$ 
 whereby $rx = xq$ and $py = yr$ by Lemma \ref{IdemFillingMorph}
 and 
 $$\mathfrak{r}(\d^{\prime}) = [X^{\prime} \overset{x^{\prime}}{\longrightarrow} Y^{\prime} \overset{y^{\prime}}{\longrightarrow} Z^{\prime}],$$
 $$\mathfrak{r}(\d^{\prime}) = [(X^{\prime},q^{\prime}) \overset{x^{\prime}q^{\prime}}{\longrightarrow} (Y^{\prime},r^{\prime}) \overset{p^{\prime}y^{\prime}}{\longrightarrow} (Z^{\prime},p^{\prime})],$$
 whereby $r^{\prime}x^{\prime} = x^{\prime}q^{\prime}$ and $p^{\prime}y^{\prime} = y^{\prime}r^{\prime}$ by Lemma \ref{IdemFillingMorph}.
 Suppose we have the following commutative diagram in $\tilde{\mathscr{C}}$. Note that we have that $q^{\prime}a = aq =a$ and $r^{\prime}b = br =b$ by the definition of morphisms in $\tilde{\mathscr{C}}$.
 
\begin{equation}\label{et3diag}
 \begin{tikzcd}[column sep=3.0em,row sep=3.0em]
(X,q) \arrow[r, "xq"] \arrow[d,"a"]
& (Y,r) \arrow[r,"py"] \arrow[d,"b"]
& (Z,p) & \text{} \\
(X^{\prime},q^{\prime}) \arrow[r, "x^{\prime}q^{\prime}"]
& (Y^{\prime},r^{\prime}) \arrow[r,"p^{\prime}y^{\prime}"]
& (Z^{\prime},p^{\prime})  &\text{}
\end{tikzcd}
\end{equation}
We then have the following diagram in $\mathscr{C}$. 
\begin{equation}\label{et3Propdiag2}
\begin{tikzcd}[column sep=3.5em,row sep=3.5em]
X \arrow[r, "x"] \arrow[d,"a"]
& Y \arrow[r,"y"] \arrow[d,"r^{\prime}br"]
& Z \arrow[r, dashed, "p^{*}q_{*}\varepsilon"] \arrow[d, dashed, "c"] & \text{} \\
X^{\prime} \arrow[r, "x^{\prime}"]
& Y^{\prime} \arrow[r,"y^{\prime}"]
& Z^{\prime} \arrow[r, dashed, "p^{\prime *}q^{\prime}_{*}\varepsilon^{\prime}"] &\text{}
\end{tikzcd}
\end{equation}

Using the above relations, we have that $$r^{\prime}brx = r^{\prime}bxq = r^{\prime}x^{\prime}q^{\prime}a = r^{\prime}r^{\prime}x^{\prime}a = r^{\prime}x^{\prime}a = x^{\prime}q^{\prime}a = x^{\prime}a$$
hence the left square of diagram (\ref{et3Propdiag2}) commutes.  Since $\mathscr{C}$ is an extriangulated category, there exists $c \colon Z \rightarrow Z^{\prime}$ such that the diagram commutes and $a_{*}(p^{*}q_{*}\varepsilon) = c^{*}(p^{\prime *}q^{\prime}_{*}\varepsilon^{\prime}).$ 

Consider the morphism $p^{\prime}cp \colon (Z,p) \rightarrow (Z,p^{\prime})$, we have that 
$$ p^{\prime}c(pp)y =p^{\prime}c(py) = p^{\prime}(cy)r = p^{\prime}y^{\prime}r^{\prime}b(rr)=p^{\prime}y^{\prime}(r^{\prime}br)=p^{\prime}y^{\prime}b,$$
 so $p^{\prime}cp$ makes diagram (\ref{et3diag}) commute.

 We also have that, $$(p^{\prime}cp)^{*}(p^{\prime *}q_{*}^{\prime} \varepsilon^{\prime})  = p^{*}c^{*}p^{\prime *}p^{\prime *}q_{*}^{\prime} \varepsilon^{\prime}  = p^{*}c^{*}p^{\prime *}q_{*}^{\prime} \varepsilon^{\prime}  =p^{*}a_{*}p^{*}q_{*}\varepsilon =a_{*}p^{*}p^{*}q_{*}\varepsilon=a_{*}p^{*}q_{*}\varepsilon $$  therefore we have a morphism of $\mathbb{F}$-extensions $(a,p^{\prime}cp) \colon \d \rightarrow \d^{\prime}$, as required. This verifies (ET3). The proof for $\text{(ET3)}^{\text{op}}$ is dual. 
\end{proof}
\end{proposition}

Before we can prove that $\tilde{\mathscr{C}}$ satisfies (ET4) and (ET4)$^\text{op}$. We first need to prove the upcoming statements, which will play an important part in our proof of (ET4) and (ET4)$^{\text{op}}$.

\begin{lemma}\label{rightClosed} Let $\d$ be an extension in $\mathbb{F}((Z,p),(X,q))$ where $\mathfrak{s}(\d)=[ X \overset{x}{\longrightarrow} Y \overset{y}{\longrightarrow} Z]$ and $\mathfrak{r}(\d) = [(X,q) \overset{xq}{\longrightarrow} (Y,r) \overset{py}{\longrightarrow} (Z,p)]$. Then the following sequences of natural transformations are exact. 
\begin{center}
\begin{tikzcd}
 \mathbb{F}(-,(X,q)) \arrow[r, Rightarrow, "\mathbb{F}(-{,}xq)"]
& \mathbb{F}(-,(Y,r)) \arrow[r, Rightarrow, "\mathbb{F}(-{,}py)"]
& \mathbb{F}(-,(Z,p))
\end{tikzcd}
\end{center}

\begin{center}
\begin{tikzcd}
 \mathbb{F}((Z,p),-) \arrow[r, Rightarrow, "\mathbb{F}(py{,}-)"]
& \mathbb{F}((Y,r),-) \arrow[r, Rightarrow, "\mathbb{F}(xq{,}-)"]
& \mathbb{F}((X,q),-)
\end{tikzcd}
\end{center}

\begin{proof}
Let $(A,e)$ be any object in $\tilde{\mathscr{C}}$. We need to show that the following sequence in $Ab$ is exact. 
\begin{center}
\begin{tikzcd}[column sep=3.5em,row sep=3.5em]
 \mathbb{F}((A,e),(X,q)) \arrow[r, "\mathbb{F}((A{,}e){,}xq)"]
& \mathbb{F}((A,e),(Y,r)) \arrow[r, "\mathbb{F}((A{,}e){,}py)"]
& \mathbb{F}((A,e),(Z,p)).
\end{tikzcd}
\end{center}
Take $\theta \in \mathbb{F}((A,e),(Y,r))$. Recall that, by definition, we have that $e^{*}\theta = r_{*}\theta = \theta$.  Suppose that $\mathbb{F}((A,e),py)(\theta) = (py)_{*}\theta = 0$. Recall that, by construction $py = yr$, hence we have that $0 = (py)_{*}\theta = (yr)_{*}\theta = y_{*}(r_{*}\theta)$. In particular we have that $r_{*}\theta = \theta \in \text{ker}(\mathbb{E}(A,y))$. By Proposition \ref{longExact}, the following sequence is exact in $Ab$.
\begin{center}
\begin{tikzcd}
 \mathbb{E}(A,X) \arrow[r,"\mathbb{E}(-{,}x)"]
& \mathbb{E}(A,Y) \arrow[r, "\mathbb{E}(-{,}y)"]
& \mathbb{E}(A,Z)
\end{tikzcd}
\end{center}
Therefore we have that $r_{*}\theta = \theta = x_{*}(\sigma)$ for some $\sigma \in \mathbb{E}(A,X)$. Observe the following, 
$$ \theta = e^{*}\theta = e^{*}x_{*}\sigma = x_{*}(e^{*}\sigma),$$
thus
$$ (xq)_{*}(e^{*}\sigma) =  (rx)_{*}(e^{*}\sigma) = r_{*}( x_{*}e^{*}\sigma) = r_{*}(\theta) = \theta.$$
Therefore we have that $(xq)_{*}(q_{*}e^{*}\sigma) = \theta$. In other words, $\theta \in \text{im}(\mathbb{F}((A,e),xq))$ and in particular, ker$(\mathbb{F}((A,e),py)) \subseteq \text{im}(\mathbb{F}((A,e),xq)).$

Now take $\sigma \in \mathbb{F}((A,e),(X,q))$, then $$(py)_{*}((xq)_{*}\sigma) = (py)_{*}(x_{*}q_{*}\sigma) =p_{*}(yx)_{*}(q_{*}\sigma) = 0$$ since $y \circ x = 0$ by Lemma \ref{weakKernelCokernel}. Hence $\text{im}(\mathbb{F}((A,e),xq)) \subseteq  \text{ker}(\mathbb{F}((A,e),py))$.

The proof of the dual statement is dual. This completes the proof. 
\end{proof}
\end{lemma}

\begin{remark}
Since $\tilde{\mathscr{C}}$ satisfies (ET3) and (ET3)$^{\text{op}}$ we have that $\tilde{\mathscr{C}}$ satisfies Proposition $\ref{et3Equiv}$. By Lemma \ref{rightClosed}, we see that $\tilde{\mathscr{C}}$ induces long exact sequences as in Proposition \ref{longExact} without requiring that $\tilde{\mathscr{C}}$ is an extriangulated category as a priori.
\end{remark}

\begin{corollary} Let $\d$ be an extension in $\mathbb{F}((Z,p),(X,q))$ where $\mathfrak{s}(\d)=[ X \overset{x}{\longrightarrow} Y \overset{y}{\longrightarrow} Z]$ and $\mathfrak{r}(\d) = [(X,q) \overset{xq}{\longrightarrow} (Y,r) \overset{py}{\longrightarrow} (Z,p)]$. Suppose that $(X,q) \overset{u}{\longrightarrow} (W,s) \overset{v}{\longrightarrow} (Z,p)$ is another sequence realising $\d$ as an $\mathbb{F}$-extension. Then the following sequences of natural transformations are exact.
\begin{center}
\begin{tikzcd}
 \mathbb{F}(-,(X,q)) \arrow[r, Rightarrow, "\mathbb{F}(-{,}u)"]
& \mathbb{F}(-,(W,s)) \arrow[r, Rightarrow, "\mathbb{F}(-{,}v)"]
& \mathbb{F}(-,(Z,p))
\end{tikzcd}
\end{center}

\begin{center}
\begin{tikzcd}
 \mathbb{F}((Z,p),-) \arrow[r, Rightarrow, "\mathbb{F}(v{,}-)"]
& \mathbb{F}((W,s),-) \arrow[r, Rightarrow, "\mathbb{F}(u{,}-)"]
& \mathbb{F}((X,q),-)
\end{tikzcd}
\end{center}

\begin{proof} Since $(X,q) \overset{u}{\longrightarrow} (W,s) \overset{v}{\longrightarrow} (Z,p)$ realises $\d$, there is an equivalence, 
\begin{center}
\begin{tikzcd}
(X,q) \arrow[r,"xq"] \arrow[d,equal]
& (Y,r) \arrow[r,"py"] \arrow[d,"f"]
& (Z,p) \arrow[d,equal]
\\
(X,q) \arrow[r,"u"]
& (W,s) \arrow[r,"v"]
& (Z,p)
\end{tikzcd}
\end{center}
Then, for any object $(A,e) \in \tilde{\mathscr{C}}$ we have the following commutative diagram, where by Lemma \ref{rightClosed}, the top row is exact. 
\begin{center}
\begin{tikzcd}
\mathbb{F}((A{,}e) {,} (X,q)) \arrow[r,"(xq)_{*}"] \arrow[d,equal]
& \mathbb{F}((A{,}e) {,} (Y,r)) \arrow[r,"(py)_{*}"] \arrow[d,"f_{*}"]
& \mathbb{F}((A{,}e) {,} (Z,p)) \arrow[d,equal]
\\
\mathbb{F}((A{,}e) {,} (X,q)) \arrow[r,"u_{*}"]
& \mathbb{F}((A{,}e) {,}(W,s)) \arrow[r,"v_{*}"]
& \mathbb{F}((A{,}e) {,}(Z,p))
\end{tikzcd}
\end{center}
From the above commutative diagram, it is easy to see that the bottom row is also exact.
\end{proof}
\end{corollary}

The following proposition is an analogue of Proposition \ref{mappingCone} in $\tilde{\mathscr{C}}$. Remarkably, we are able to prove the statement of the following proposition without requiring $\tilde{\mathscr{C}}$ to be extriangulated unlike in the statement of Proposition \ref{mappingCone}. We only require that $(\tilde{\mathscr{C}}, \mathbb{F},\mathfrak{r})$ satisfies axioms (ET1) and (ET2). 

\begin{proposition}\label{FMappingCone}
Let $\d = p^{*}q_{*}\varepsilon \in \mathbb{F}((C,p),(A,q))$ be an $\mathbb{F}$-extension where $$\mathfrak{s}(\d)=[ A \overset{x}{\longrightarrow} B \overset{y}{\longrightarrow} C] \text{ and } \mathfrak{r}(\d) = [(A,q) \overset{xq}{\longrightarrow} (B,r) \overset{py}{\longrightarrow} (C,p)].$$
Let $h \colon (E,w) \longrightarrow (C,p)$ be any morphism and suppose
$$\mathfrak{s}(h^{*}\d)=[ A \overset{d}{\longrightarrow} D \overset{e}{\longrightarrow} E] \text{ and } \mathfrak{r}(h^{*}\d) = [(A,q) \overset{dq}{\longrightarrow} (D,s) \overset{we}{\longrightarrow} (E,w)].$$
Then there exists a morphism $g \colon (D,s) \longrightarrow (B,r)$ such that $(1_{(A,q)},h) \colon h^{*}\d \rightarrow \d$ is realised by $(1_{(A,q)}, g, h).$
Moreover $$\mathfrak{r}((dq)_{*}\d) = [(D,s) \overset{\begin{pmatrix}
-es \\ gs
\end{pmatrix}}{\longrightarrow} (E,w) \oplus (B,r) \overset{\begin{pmatrix}
h & py
\end{pmatrix}}{\longrightarrow} (C,p)].$$

\begin{proof}
We apply Proposition \ref{mappingCone} to \begin{tikzcd} 
A \arrow[r, "x"]
& B \arrow[r,"y"]
& C \arrow[r, dashed, "\d"] &\text{}
\end{tikzcd}, the morphism $h \colon E \rightarrow C$ and \begin{tikzcd} 
A \arrow[r, "d"]
& D \arrow[r,"e"]
& E \arrow[r, dashed, "h^{*}\d"] &\text{.}
\end{tikzcd} 
Then there is a morphism $\bar{g} \colon D \rightarrow B$ such that the following diagram commutes
\begin{equation}\label{finalPropDiag1}
\begin{tikzcd} 
A \arrow[r,"d"] \arrow[d,equal]
& D \arrow[r,"e"] \arrow[d,"\bar{g}"]
& E \arrow[r,dashed, "h^{*}\d"] \arrow[d,"h"]
& \text{} 
\\
A \arrow[r, "x"] 
& B \arrow[r,"y"] 
& C \arrow[r, dashed, "\d"] 
&\text{}
\end{tikzcd}
\end{equation}

and that 
\begin{center}
\begin{tikzcd} 
D \arrow[r, "\begin{pmatrix}
-e \\ \bar{g}
\end{pmatrix}"]
& E \oplus B \arrow[r,"\begin{pmatrix}
h \hspace{0.3cm} y
\end{pmatrix}"]
&C \arrow[r, dashed, "d_{*}\d"] &\text{.}
\end{tikzcd}
\end{center}

Since $h \colon (E,w) \longrightarrow (C,p)$ is a morphism in $\tilde{\mathscr{C}}$ we have that $h=hw$, therefore $$h^{*}\d = w^{*}h^{*}\d = w^{*}h^{*}p^{*}q_{*}\varepsilon = w^{*}q_{*}(h^{*}p^{*}\varepsilon).$$ In other words $h^{*}\d \in \mathbb{F}((E,w),(A,q))$, so we have that $$\mathfrak{r}(h^{*}\d) = [(A,q) \overset{dq}{\longrightarrow} (D,s) \overset{we}{\longrightarrow} (E,w)],$$
where $s \colon D \rightarrow D$ is an idempotent morphism such that $dq=sd$ and $we=es$. Consider the following diagram. 
\begin{equation}\label{finalPropDiag2}
\begin{tikzcd} 
(A,q) \arrow[r,"dq"] \arrow[d,equal]
& (D,s) \arrow[r,"we"] \arrow[d,"g=\bar{rgs}"]
& (E,w) \arrow[r,dashed, "h^{*}\d"] \arrow[d,"h"]
& \text{} 
\\
(A,q) \arrow[r, "xq"] 
& (B,r) \arrow[r,"py"] 
& (C,p) \arrow[r, dashed, "\d"] 
&\text{}
\end{tikzcd}
\end{equation}
By diagram (\ref{finalPropDiag1}) and the relations $h=hw=ph$, we can observe that $$(r\bar{g}s)dq = r\bar{g}(dq) = r(\bar{g}d)q = rxq = xq \text{ and}$$
$$py(r\bar{g}s) = (pyr)\bar{g}s = (py)\bar{g}s = p(y\bar{g})s = p(he)s = (ph)es = h(es) = hwe.$$
Therefore diagram (\ref{finalPropDiag2}) commutes and $(1_{(A,q)},h)$ is realised by $(1_{(A,q)}, g, h).$

Since $dq = sd$, we have that $$(dq)_{*}\d = (sd)_{*}p^{*}q_{*}\varepsilon = s_{*}d_{*}p^{*}q_{*}\varepsilon = s_{*}p^{*}(d_{*}q_{*}\varepsilon).$$
That is to say, $(dq)_{*}\d \in \mathbb{F}((C,p),(D,s))$. By definition $(E,w) \oplus (B,r) = (E \oplus B, w \oplus r)$, where $w \oplus r = \begin{pmatrix}
w & 0 \\
0 & r
\end{pmatrix}$, which we observe is an idempotent morphism. Also observe that, 
$$ \begin{pmatrix}
w & 0 \\
0 & r
\end{pmatrix} \begin{pmatrix}
-e\\
g
\end{pmatrix} = \begin{pmatrix}
-we\\
rg
\end{pmatrix} = \begin{pmatrix}
-es\\
gs
\end{pmatrix} = \begin{pmatrix}
-e\\
g
\end{pmatrix}s \text{ and}$$

$$ \begin{pmatrix}
h & y
\end{pmatrix} \begin{pmatrix}
w & 0\\
0 & r
\end{pmatrix} = 
\begin{pmatrix}
hw & yr 
\end{pmatrix} = \begin{pmatrix}
h & py
\end{pmatrix}.$$ 
Therefore $$\mathfrak{r}((dq)_{*}\d) = [(D,s) \overset{\begin{pmatrix}
-es \\ gs
\end{pmatrix}}{\longrightarrow} (E,w) \oplus (B,r) \overset{\begin{pmatrix}
h & py
\end{pmatrix}}{\longrightarrow} (C,p)],$$
as required. 
\end{proof}
\end{proposition}

\begin{corollary}\label{mappingCorollary}
Let $\varepsilon =  p^{*}q_{*}\sigma \in \mathbb{F}((C,p),(A,q))$ and $\delta = t^{*}q_{*}\theta \in \mathbb{F}((Z,t),(A,q))$ be $\mathbb{F}$-extensions where 
$$\mathfrak{s}(\varepsilon) = [A \overset{a}{\longrightarrow} B \overset{b}{\longrightarrow} C] \text{ and } \mathfrak{r}(\varepsilon) = [(A,q) \overset{aq}{\longrightarrow} (B,r) \overset{pb}{\longrightarrow} (C,p)]$$
and 
$$ \mathfrak{s}(\d) = [A \overset{x}{\longrightarrow} Y \overset{y}{\longrightarrow} Z] \text{ and } \mathfrak{r}(\d) = [(A,q) \overset{xq}{\longrightarrow} (Y,s) \overset{ty}{\longrightarrow} (Z,t)].$$
Suppose we have the following diagram, where the left square commutes i.e. $uaq=xq$. 
\begin{center}
\begin{tikzcd} 
(A,q) \arrow[r, "aq"] \arrow[d,equal]
& (B,r) \arrow[r,"pb"] \arrow[d,"u"]
& (C,p) \arrow[r, dashed, "\varepsilon"] &\text{}
\\
(A,q) \arrow[r, "xq"]
& (Y,s) \arrow[r,"ty"]
& (Z,t) \arrow[r, dashed, "\d"] &\text{}
\end{tikzcd}
\end{center}
Then there exists a morphism $w \colon (C,p) \rightarrow (Z,t)$ in $\tilde{\mathscr{C}}$ such $wpb = tyu, w^{*}\d = \varepsilon$ and that the following is an $\mathbb{F}$-triangle,

\begin{center}
\begin{tikzcd} 
(B,r) \arrow[r, "\begin{pmatrix}
-pb \\ u
\end{pmatrix}"]
& (C,p) \oplus (Y,s) \arrow[r,"\begin{pmatrix}
w \hspace{0.3cm} ty
\end{pmatrix}"]
& (Z,t) \arrow[r, dashed, "(aq)_{*}\d"] &\text{.}
\end{tikzcd}
\end{center}

\begin{proof} This statement is the analogue of the statement of \cite[Proposition 3.5]{herschend2017n} (Corollary \ref{mappingCone2}). The proof of \cite[Proposition 3.5]{herschend2017n} is a consequence of (ET2) and Proposition \ref{mappingCone}, or axioms (R0) and (EA2) respectively in the language of \cite{herschend2017n}. We have shown that $(\tilde{\mathscr{C}},\mathbb{F},\mathfrak{r})$ satisfies (ET2) in Proposition \ref{addreal} and Proposition \ref{FMappingCone} shows that $(\tilde{\mathscr{C}},\mathbb{F},\mathfrak{r})$ satisfies Proposition \ref{mappingCone}. Hence the statement of the corollary follows, by using an argument as in \cite{herschend2017n}.
\end{proof}

\end{corollary}

So far we have shown that the triple $(\tilde{\mathscr{C}},\mathbb{F},\mathfrak{r})$ satisfies the axioms (ET1), (ET2), (ET3) and (ET3)$^{\text{op}}$. We are now in a position to prove axioms (ET4) and (ET4)$^{\text{op}}$. 

\begin{proposition}\label{propet4}The triple $(\tilde{\mathscr{C}}, \mathbb{F}, \mathfrak{r)}$ satisfies the axioms (ET4) and $\text{(ET4)}^{\text{op}}$.
\begin{proof}
Let $(D,p), (A,q), (F,t) \text{ and } (B,r)$ be objects in $\tilde{\mathscr{C}}$ and let $\d \in \mathbb{F}((D,p),(A,q))$ and $\d^{\prime} \in \mathbb{F}((F,t),(B,r))$ be $\mathbb{F}$-extensions with
 $$\mathfrak{s}(\d) =  [A \overset{f}{\longrightarrow} B \overset{f^{\prime}}{\longrightarrow} D],$$ and 
 $$\mathfrak{s}(\d^{\prime}) = [B \overset{g}{\longrightarrow} C \overset{g^{\prime}}{\longrightarrow} F]$$ 
 in the extriangulated category $(\mathscr{C}, \mathbb{E}, \mathfrak{s})$. 
 Then by definition
$$\mathfrak{r}(\d) =  [(A,q) \overset{fq}{\longrightarrow} (B,r) \overset{pf^{\prime}}{\longrightarrow} (D,p)],$$
for some idempotent $r \colon B \rightarrow B$ where 
\begin{equation}\label{idemEqns1} 
fq=rf  \text{ and }pf^{\prime} = f^{\prime}r 
\end{equation} 
and
$$\mathfrak{r}(\d^{\prime}) = [(B,r) \overset{gr}{\longrightarrow} (C,s) \overset{tg^{\prime}}{\longrightarrow} (F,t)],$$
for some idempotent $s \colon C \rightarrow C$ where 
\begin{equation}\label{idemEqns2}
gr = sg \text{ and } tg^{\prime} = g^{\prime}s.
\end{equation}

We must show that there exists an object $(E,w) \in \tilde{\mathscr{C}}$, an $\mathbb{F}$-extension $\d^{\prime \prime} \in \mathbb{F}((E,w),(A,q))$ such that the following diagram commutes, 
\begin{center}
\begin{tikzcd}
(A,q) \arrow[r, "fq"] \arrow[d, equal]
& (B,r) \arrow[d, "gr"] \arrow[r,"pf^{\prime}"]
& (D,p) \arrow[d,"\bar{d}"] \\
 (A,q) \arrow[r,"hq"]
& (C,s) \arrow[r,"wh^{\prime}"] \arrow[d,"tg^{\prime}"]
&(E,w) \arrow[d,"\bar{e}"] \\
 & (F,t) \arrow[r,equal] & (F,t)
\end{tikzcd}
\end{center}
and that the following compatibilities hold,
\begin{enumerate}[(1)]
\item $\mathfrak{r}((pf^{\prime})_{*}\d^{\prime}) = [(D,p) \overset{\bar{d}}{\longrightarrow} (E,w) \overset{\bar{e}}{\longrightarrow} (F,t)].$

\item $(\bar{d})^{*} \d^{\prime \prime} = \d.$

\item $(fq)_{*}\d^{\prime \prime} = (\bar{e})^{*}\d^{\prime}.$
\end{enumerate}

Since $(\mathscr{C},\mathbb{E},\mathfrak{s})$ is extriangulated we can apply (ET4) to the above $\mathbb{E}$-triangles to get an object $E$ in $\mathscr{C}$, a commutative diagram
\begin{equation}\label{et4proofdiag1}
\begin{tikzcd}
A \arrow[r, "f"] \arrow[d, equal]
& B \arrow[d, "g"] \arrow[r,"f^{\prime}"]
& D \arrow[d,"d"] \\
 A \arrow[r,"h"]
&C \arrow[r,"h^{\prime}"] \arrow[d,"g^{\prime}"]
&E \arrow[d,"e"] \\
 & F \arrow[r,equal] & F
\end{tikzcd}
\end{equation}
in $\mathscr{C}$ and an $\mathbb{E}$-extension $\d^{\prime \prime} \in \mathbb{E}(E,A)$ where $$ \mathfrak{s}(\d^{\prime \prime}) = [A \overset{h}{\longrightarrow} C \overset{h^{\prime}}{\longrightarrow} E]$$ such that the following compatibilities are satisfied:
\begin{enumerate}[(i)]
\item $\mathfrak{s}((f^{\prime})_{*}\d^{\prime}) = [D \overset{d}{\longrightarrow} E \overset{e}{\longrightarrow} F].$

\item $d^{*} \d^{\prime \prime} = \d.$

\item $f_{*}\d^{\prime \prime} = e^{*}\d^{\prime}.$
\end{enumerate}

Recall that since $\d^{\prime} \in \mathbb{F}((F,t),(B,r))$, then by definition $\d^{\prime} = t^{*}r_{*}\varepsilon^{\prime}$ for some $\varepsilon^{\prime} \in  \mathbb{E}(F,B)$. Also recall that $pf^{\prime} = f^{\prime}r$ by (\ref{idemEqns1}), so we have that 
$$ f^{\prime}_{*}\d^{\prime} = f^{\prime}_{*} t^{*}r_{*}\varepsilon^{\prime} = t^{*} f^{\prime}_{*}r_{*}\varepsilon^{\prime} = t^{*}(f^{\prime}r)_{*}\varepsilon^{\prime} = t^{*}(pf^{\prime})_{*}\varepsilon^{\prime} = t^{*}p_{*}(f^{\prime}_{*}\varepsilon^{\prime}).$$
In other words $f^{\prime}_{*}\d^{\prime} \in \mathbb{F}((F,t),(D,p))$ and so we have by definition that 
\begin{equation}\label{pbDeltaPrime}
\mathfrak{r}(f^{\prime}_{*}\d^{\prime}) = [(D,p) \overset{dp}{\longrightarrow} (E,v) \overset{te}{\longrightarrow} (F,t)],
\end{equation}
where $v \colon E \rightarrow E$ is an idempotent such that 
\begin{equation}\label{idemEqns3}
dp = vd \text{ and } te = ev.
\end{equation}

Now consider the element $\d^{\prime \prime} \in \mathbb{E}(E,A)$. We are going to show that $\d^{\prime \prime} \in \mathbb{F}((E,v),(A,q))$. Note that by the compatibility (iii), we have that $f_{*}\d^{\prime \prime} = e^{*}\d^{\prime}$. Recall the relations $fq=rf, te=ev$ and $t^{*}\d^{\prime} = \d^{\prime}$ by (\ref{idemEqns1}) and (\ref{idemEqns3}). We can then observe that 
$$(fq)_{*}\d^{\prime \prime}  = (rf)_{*}\d^{\prime \prime} = r_{*}f_{*}\d^{\prime \prime} = r_{*}e^{*}t^{*}\d^{\prime} = r_{*}(te)^{*}\d^{\prime} = r_{*}(ev)^{*}\d^{\prime} = r_{*}v^{*}(e^{*}\d^{\prime})\in \mathbb{F}((E,v),(B,r)).$$
Consider the $\mathbb{F}$-triangle, 
\begin{center}
\begin{tikzcd}
(A,q) \arrow[r,"fq"]
& (B,r) \arrow[r,"pf^{\prime}"]
& (D,p) \arrow[r,dashed,"\d"]
&\text{,}
\end{tikzcd}
\end{center}
and the following diagram arising from it. 
\begin{center}
\begin{tikzcd}

& \mathbb{F}((E{,}v) {,} (A{,}q)) \arrow[r,"(fq)_{*}"] \arrow[d,hook]
& \mathbb{F}((E{,}v) {,} (B{,}r)) \arrow[r,"(pf^{\prime})_{*}"] \arrow[d,hook]
&\mathbb{F}((E{,}v) {,} (D{,}p)) \arrow[d,hook]
\\
\tilde{\mathscr{C}}((E{,}v) {,} (D{,}p)) \arrow[r,"\d_{\#}"] \arrow[bend left =20, swap]{ur}{\d_{\#}} 
& \mathbb{E}(E,A) \arrow[r,"(fq)_{*}"]
&\mathbb{E}(E,B) \arrow[r,"(pf^{\prime})_{*}"]
&\mathbb{E}(E,D)
\end{tikzcd}
\end{center}
Note that the vertical inclusion maps are due to the fact that the $\mathbb{F}$-extension groups are subgroups of the respective $\mathbb{E}$-extension groups and the diagram commutes. By Lemma \ref{rightClosed}, the top row is exact in $Ab$. Moreover, the sequence obtained by appending the morphism $\d_{\#} \colon \tilde{\mathscr{C}}((E,v),(D,p)) \rightarrow \mathbb{F}((E,v),(A,q))$ to the top row is exact by Proposition \ref{et3Equiv} and Lemma \ref{rightClosed}.

Observe that $(fq)_{*}\d^{\prime \prime} \in \mathbb{F}((E,v),(B,r)$ and $$(pf^{\prime})_{*}((fq)_{*}\d^{\prime \prime} )= p_{*}(f^{\prime} f)_{*}q_{*}\d^{\prime \prime} = p_{*}0_{*}q_{*}\d^{\prime \prime} = 0.$$
That is to say, $(fq)_{*}\d^{\prime \prime}$ is in the kernel of $(pf^{\prime})_{*}$ in the top row. So there exists an $\mathbb{F}$-extension $\sigma \in \mathbb{F}((E,v),(A,q))$ such that $(fq)_{*}\sigma = (fq)_{*}\d^{\prime \prime}$, so $(fq)_{*}(\sigma - \d^{\prime \prime}) = 0$. That is to say $\sigma - \d^{\prime \prime}$ is in the kernel of $(fq)_{*}$ in the top row, therefore there exists $k \in \tilde{\mathscr{C}}((E,v),(D,p))$ such that $\d_{\#}(k) = \sigma - \d^{\prime \prime}$. In other words $\d^{\prime \prime} =  \sigma - \d_{\#}(k) \in \mathbb{F}((E,v),(A,q))$ as required.

Consider the solid part of the following diagram.
\begin{equation}\label{idemFillET4}
\begin{tikzcd}
A \arrow[r,"h"] \arrow[d,"q"]
& C \arrow[r,"h^{\prime}"] \arrow[d,"s"]
& E \arrow[r,dashed,"\d^{\prime \prime}"] \arrow[d,dashed,"w"]
& \text{}
\\
A \arrow[r,"h"]
& C \arrow[r,"h^{\prime}"]
& E \arrow[r,dashed,"\d^{\prime \prime}"]
& \text{}
\end{tikzcd}
\end{equation}
By (\ref{idemEqns2}) we have that, $$sh = s(gf) = (sg)f = g(rf) =(gf)q = hq,$$
therefore the solid square commutes, so by axiom (ET3) for $\mathscr{C}$, there exists a morphism $u \colon E \rightarrow E$ such that $(q,u) \colon \d^{\prime \prime} \rightarrow \d^{\prime \prime}$ is a morphism of $\mathbb{E}$-extensions realised by $(q,s,u)$. Since $q$ and $s$ are idempotent, it follows from Lemma \ref{IdemFillingMorph2} that there exists an idempotent morphism $w \colon E \rightarrow E$ such that $(q,w) \colon \d^{\prime \prime} \rightarrow \d^{\prime \prime}$ is a morphism of $\mathbb{E}$-extensions and diagram (\ref{idemFillET4}) commutes. As $$\d^{\prime \prime} = q_{*}\d^{\prime \prime} = w^{*}\d^{\prime \prime} = w^{*}q_{*}\d^{\prime \prime},$$ we have that $\d^{\prime \prime} \in \mathbb{F}((E,w),(A,q))$, and since $s \colon C \rightarrow C$ is an idempotent such that $$sh = hq \text{ and } wh^{\prime} = h^{\prime} s,$$
and furthermore 
$$ \mathfrak{s}(\d^{\prime \prime}) = [A \overset{h}{\longrightarrow} C \overset{h^{\prime}}{\longrightarrow} E],$$
we have that 
$$\mathfrak{r}(\d^{\prime \prime}) =  [(A,q) \overset{hq}{\longrightarrow} (C,s) \overset{wh^{\prime}}{\longrightarrow} (E,w)].$$
Applying Corollary \ref{mappingCorollary} to the following solid commutative diagram,
\begin{equation}\label{topRow}
\begin{tikzcd}
(A,q) \arrow[r,"fq"] \arrow[d,equal]
& (B,r) \arrow[r,"pf^{\prime}"] \arrow[d,"gr"]
& (D,p) \arrow[d, dashed, "\bar{d}"] \arrow[r,dashed, "\d"]
& \text{}
\\
(A,q) \arrow[r,"hq"] 
& (C,s) \arrow[r,"wh^{\prime}"]
& (E,w) \arrow[r,dashed, "\d^{\prime \prime}"]
& \text{}
\end{tikzcd}
\end{equation}
we get a morphism $\bar{d} \colon (D,p) \rightarrow (E,w)$ such that $\bar{d} \circ pf^{\prime} = wh^{\prime} \circ gr, \bar{d}^{*}\d^{\prime \prime} = \d^{\prime} \text{ and that}$

\begin{center}
\begin{tikzcd} 
(B,r) \arrow[r, "\begin{pmatrix}
-pf^{\prime} \\ gr
\end{pmatrix}"]
& (D,p) \oplus (C,s) \arrow[r,"\begin{pmatrix}
\bar{d} \hspace{0.3cm} wh^{\prime}
\end{pmatrix}"]
& (E,w) \arrow[r, dashed, "(fq)_{*}\d^{\prime \prime}"] &\text{,}
\end{tikzcd}
\end{center}
is an $\mathbb{F}$-triangle.

Applying Corollary \ref{mappingCorollary} to the following solid commutative diagram,
\begin{equation}\label{botRow}
\begin{tikzcd}[column sep=3.5em,row sep=3.5em]
(B,r) \arrow[r, "\begin{pmatrix}
-pf^{\prime} \\ gr
\end{pmatrix}"] \arrow[d,equal]
& (D,p) \oplus (C,s) \arrow[r,"\begin{pmatrix}
\bar{d} \hspace{0.3cm} wh^{\prime}
\end{pmatrix}"] \arrow[d,"\begin{pmatrix}
0 \hspace{0.3cm} 1
\end{pmatrix}"]
& (E,w) \arrow[r, dashed, "(fq)_{*}\d^{\prime \prime}"] \arrow[d, dashed,"\bar{e}"] &\text{}
\\
(B,r) \arrow[r, "gr"]
& (C,s) \arrow[r,"tg^{\prime}"]
& (F,t) \arrow[r, dashed, "\d^{\prime}"] &\text{}
\end{tikzcd}
\end{equation}
we get a morphism $\bar{e} \colon (E,w) \rightarrow (F,t)$ such that $\bar{e} \circ \begin{pmatrix}
\bar{d} \hspace{0.3cm} wh^{\prime}
\end{pmatrix} = tg^{\prime} \circ \begin{pmatrix}
0 \hspace{0.3cm} 1
\end{pmatrix} , \bar{e}^{*}\d^{\prime} = (fq)_{*}\d^{\prime \prime} \text{ and that,}$

\begin{center}
\begin{tikzcd}[column sep=3.5em,row sep=3.5em]
(D,p) \oplus (C,s) \arrow[r, "\begin{pmatrix}
-\bar{d} \hspace{0.3cm} -wh^{\prime} \\ 
0 \hspace{0.4cm} 1
\end{pmatrix}"]
& (E,w) \oplus (C,s) \arrow[r,"\begin{pmatrix}
\bar{e} \hspace{0.3cm} tg^{\prime}
\end{pmatrix}"]
& (F,t) \arrow[r, dashed, "\begin{pmatrix}
-pf^{\prime} \\ gr
\end{pmatrix}_{*}\d^{\prime}"] &\text{,}
\end{tikzcd}
\end{center}
is an $\mathbb{F}$-triangle.

By Proposition \ref{dirSum} applied to the above $\mathbb{F}$-triangle, the following is an $\mathbb{F}$-triangle. 
\begin{center}
\begin{tikzcd}
(D,p) \arrow[r,"-\bar{d}"] 
& (E,w) \arrow[r,"\bar{e}"]
& (F,t) \arrow[r,dashed,"(-pf^{\prime})_{*}\d^{\prime}"]
& \text{}
\end{tikzcd}
\end{center}
This $\mathbb{F}$-triangle is isomorphic to
\begin{equation}\label{lastFTriangle}
\begin{tikzcd}
(D,p) \arrow[r,"\bar{d}"] 
& (E,w) \arrow[r,"\bar{e}"]
& (F,t) \arrow[r,dashed,"(pf^{\prime})_{*}\d^{\prime}"]
& \text{,}
\end{tikzcd}
\end{equation}
Hence, by Corollary \ref{closedUnderIso2}, we have that (\ref{lastFTriangle}) is an $\mathbb{F}$-triangle, so
$$\mathfrak{r}((pf^{\prime})_{*}\d^{\prime}) = [(D,p) \overset{\bar{d}}{\longrightarrow} (E,w) \overset{\bar{e}}{\longrightarrow} (F,t) ].$$

Now consider  the following diagram.
\begin{equation}\label{et4proofdiag2}
\begin{tikzcd}
(A,q) \arrow[r, "fq"] \arrow[d, equal]
& (B,r) \arrow[d, "gr"] \arrow[r,"pf^{\prime}"]
& (D,p) \arrow[d,"\bar{d}"] \\
 (A,q) \arrow[r,"hq"]
& (C,s) \arrow[r,"wh^{\prime}"] \arrow[d,"tg^{\prime}"]
&(E,w) \arrow[d,"\bar{e}"] \\
 & (F,t) \arrow[r,equal] & (F,t)
\end{tikzcd}
\end{equation}
Using the relations arising from the commutative diagram (\ref{topRow}), we have that the top squares of the above diagram commute and that $\bar{d}^{*}\d^{\prime \prime } = \d$. From the relations arising from the commutative diagram (\ref{botRow}) we have that the bottom right square of the above diagram commutes and that $(fq)_{*}\d^{\prime \prime} = \bar{e}^{*}\d^{\prime}.$ 

To conclude we have shown that there exists an object $(E,w) \in \tilde{\mathscr{C}}$, an $\mathbb{F}$-extension $\d^{\prime \prime} \in \mathbb{F}((E,w),(A,q))$ such that the following diagram commutes, 
\begin{center}
\begin{tikzcd}
(A,q) \arrow[r, "fq"] \arrow[d, equal]
& (B,r) \arrow[d, "gr"] \arrow[r,"pf^{\prime}"]
& (D,p) \arrow[d,"\bar{d}"] \\
 (A,q) \arrow[r,"hq"]
& (C,s) \arrow[r,"wh^{\prime}"] \arrow[d,"tg^{\prime}"]
&(E,w) \arrow[d,"\bar{e}"] \\
 & (F,t) \arrow[r,equal] & (F,t)
\end{tikzcd}
\end{center}
and that the following compatibilities hold.
\begin{enumerate}[(1)]
\item $\mathfrak{r}((pf^{\prime})_{*}\d^{\prime}) = [(D,p) \overset{\bar{d}}{\longrightarrow} (E,w) \overset{\bar{e}}{\longrightarrow} (F,t)].$

\item $(\bar{d})^{*} \d^{\prime \prime} = \d.$

\item $(fq)_{*}\d^{\prime \prime} = (\bar{e})^{*}\d^{\prime}.$
\end{enumerate}

This completes the proof of (ET4). The proof of (ET4)$^{\text{op}}$ is dual. 
\end{proof}
\end{proposition}

Having shown that $(\tilde{\mathscr{C}},\mathbb{E},\mathfrak{r})$ satisfies (ET4) and (ET4)$^{\text{op}}$, we can now conclude that $(\tilde{\mathscr{C}},\mathbb{E},\mathfrak{r})$ is an extriangulated category. Recall that there is a fully faithful additive functor $i_{\mathscr{C}} \colon \mathscr{C} \rightarrow \tilde{\mathscr{C}}$ defined as follows. For an object $A$ of $\mathscr{C}$, we have that $i_{\mathscr{C}}(A) = (A,1_A)$ and for a morphism $f$ in $\mathscr{C}$, we have that $i_{\mathscr{C}}(f)=f$. We will show that this functor is an extriangulated functor in the sense of \cite[Definition 2.31]{BTShah}. In particular, the functor $i_{\mathscr{C}}$ preserves the extriangulated structure of $\mathscr{C}$.

\begin{proposition}\label{iExtFunct} Let $\mathscr{C}$ be an extriangulated category and $\tilde{\mathscr{C}}$ be its idempotent completion. Then the functor $i_{\mathscr{C}} \colon \mathscr{C} \rightarrow \tilde{\mathscr{C}}$ is an extriangulated functor. 
\begin{proof}
It is easy to see that the functor $i_{\mathscr{C}}$ is a covariant additive functor. So all that is left is to define a natural transformation $$\Gamma = \{\Gamma_{(C,A)}\}_{(C,A) \in \mathscr{C}^{\text{op}} \times \mathscr{C}} \colon \mathbb{E} \Rightarrow \mathbb{F}(i_{\mathscr{C}}^{\text{op}}-,i_{\mathscr{C}}-)$$ 
of functors $\mathscr{C}^{\text{op}} \times \mathscr{C} \rightarrow Ab$, such that for any $\mathbb{E}$-extension $\d$ if $$\mathfrak{s}(\d) = [X \overset{x}{\longrightarrow} Y \overset{y}{\longrightarrow} Z]$$ then $$\mathfrak{r}(\Gamma_{(Z,X)}(\d)) = [i_\mathscr{C}(A) \overset{i_{\mathscr{C}}(x)}{\longrightarrow} i_\mathscr{C}(B) \overset{i_\mathscr{C}(y)}{\longrightarrow} i_\mathscr{C}(C)].$$
First note that by definition, $\mathbb{F}((C,1),(A,1)) = \mathbb{E}(C,A).$ So given a pair of objects $A,C$ in $\mathscr{C}$ we define $\Gamma_{(C,A)} \colon \mathbb{E}(C,A) \rightarrow \mathbb{F}((C,1),(A,1))$ by setting $\Gamma_{(C,A)}(\d) = \d$ for all $\d \in \mathbb{E}(C,A)$.
Given a morphism $(f,g) \colon (C,A) \rightarrow (Z,X)$ in $\mathscr{C}^{\text{op}} \times \mathscr{C}$, consider the following diagram:
\begin{center}
\begin{tikzcd}[column sep=3.5em,row sep=3.5em]
\mathbb{E}(C,A) \arrow[r, " \Gamma_{(C,A)}"] \arrow[d, " \mathbb{E}(f{,}g)" left]
& \mathbb{F}((C,1),(A,1)) \arrow[d, "\mathbb{F}(f{,}g)"] \\
\mathbb{E}(X,Z) \arrow[r, " \Gamma_{(Z,X)}"] 
& \mathbb{F}((X,1),(Z,1)).
\end{tikzcd}
\end{center}
For $\d \in \mathbb{E}(C,A)$, we have that $\mathbb{E}(f,g)(\d) = f^{*}g_{*}\d$. Therefore $\Gamma_{(Z,X)}(\mathbb{E}(f,g)(\d))=f^{*}g_{*}\d.$ On the other hand, $\Gamma_{(C,A)}(\d) =\d$ and $\mathbb{F}(f,g)(\d)=f^{*}g_{*}\d$. So the diagram commutes. 

Let $A,C$ be an objects in $\mathscr{C}$ and $\d$ be any extension in $\mathbb{E}(C,A)$. Then by definition, we have that $\Gamma_{(C,A)}(\d) = \d \in \mathbb{F}((C,1),(A,1))$. Suppose 
$$ \mathfrak{s}(\d) = [A \overset{x}{\longrightarrow} B \overset{y}{\longrightarrow} C].$$  
Then
\begin{align*}
\mathfrak{r}(\Gamma_{(C,A)}(\d)) &= [(A,1)  \overset{x}{\longrightarrow} (B,1) \overset{y}{\longrightarrow} (C,1)] \\
&= [i_\mathscr{C}(A) \overset{i_{\mathscr{C}}(x)}{\longrightarrow} i_\mathscr{C}(B) \overset{i_\mathscr{C}(y)}{\longrightarrow} i_\mathscr{C}(C)].
\end{align*}
So we conclude that $i_{\mathscr{C}} \colon \mathscr{C} \rightarrow \tilde{\mathscr{C}}$ is in fact an extriangulated functor as required.
\end{proof}
\end{proposition}

\begin{theorem}\label{IdempCompExtri}
Let $(\mathscr{C}, \mathbb{E}, \mathfrak{s})$ be an extriangulated category. Let $\tilde{\mathscr{C}}$ be the idempotent completion of $\mathscr{C}$. Then $\tilde{\mathscr{C}}$ is extriangulated. Moreover, the embedding $i_{\mathscr{C}} \colon \mathscr{C} \rightarrow \tilde{\mathscr{C}}$ is an extriangulated functor.
\begin{proof}
This follows from the following Propositions \ref{biadd1}, \ref{biadd2}, \ref{addreal}, \ref{propet3}, \ref{propet4} and \ref{iExtFunct}.
\end{proof}
\end{theorem}

\subsection{Weak idempotent completion.}

\begin{definition}\cite[1.12]{Neeman}.
Let $\mathscr{A}$ be a small additive category. The \textit{weak idempotent completion} of $\mathscr{A}$ is denoted by $\hat{\mathscr{A}}$ and is defined as follows. The objects of $\hat{\mathscr{A}}$ are the pairs $(A,p)$ where $A$ is an object of $\mathscr{A}$ and $p \colon A \rightarrow A$ is an idempotent factoring as $p=cr$ for some retraction $r \colon A \rightarrow X$ and some section $c \colon X \rightarrow A$ with $rc=1_X$ (i.e. $p$ is a split idempotent). A morphism in $\hat{\mathscr{A}}$ from $(A,p)$ to $(B,q)$ is a morphism $\sigma \colon A \rightarrow B \in \mathscr{A}$ such that $\sigma p = q \sigma = \sigma$. 
\end{definition}
There is fully faithful additive functor $j_{\mathscr{A}} \colon \mathscr{A} \rightarrow \hat{\mathscr{A}}$ from $\mathscr{A}$ to its weak idempotent completion defined as follows. For an object $A$ in $\mathscr{A}$, we have that $j_{\mathscr{A}}(A) = (A,1_A)$, and for a morphism $f$ in $\mathscr{C}$, we have that $j_{\mathscr{A}}(f)=f$.

The following lemma is key in proving the main theorem for this subsection. It is an analogue of Lemma \ref{IdemFillingMorph} where we replace the idempotent morphisms with split idempotent morphisms.

\begin{lemma}\label{SplitIdemFillingMorph} Let $(\mathscr{A}, \mathbb{G},\mathfrak{t})$ be a triple satisfying (ET1) and (ET2). Let $A,B \text{ and }C$ be objects of $\mathscr{A}$. Let $\d$ be an extension in $\mathbb{G}(C,A)$ with $\mathfrak{t}(\d) = [ A \overset{a}{\longrightarrow} B \overset{b}{\longrightarrow} C]$. Let $(e,f) \colon \d \rightarrow \d$  be a morphism of $\mathbb{G}$-extensions where $e \colon A \rightarrow A$ and $f \colon C \rightarrow C$ are idempotent morphisms that split.  Then there exists an idempotent morphism $g \colon B \rightarrow B$ that splits such that the triple $(e,g,f)$ realises the morphism of $\mathbb{G}$-extensions $(e,f)$. 
\begin{center}
\begin{tikzcd}
A \arrow[r,"a"] \arrow[d,"e"]
& B \arrow[r,"b"] \arrow[d, dashed, "g"]
& C \arrow[d,"f"]
\\
A \arrow[r,"a"]
& B \arrow [r,"b"]
& C
\end{tikzcd}
\end{center}
\begin{proof}
Since $e$ splits, there is an object $X \in \mathscr{A}$ and morphisms $e_{2} \colon A \rightarrow X$ and $e_{1} \colon X \rightarrow A$ such that $e=e_{1}e_{2}$ and $e_{2}e_{1}=1_{X}$. Likewise, for $f$ there is an object $Z \in \mathscr{A}$ and morphisms $f_{2} \colon C \rightarrow Z$ and $f_{1} \colon Z \rightarrow C$ such that $f=f_{1}f_{2}$ and $f_{2}f_{1}=1_{Z}$. 

Suppose that $\mathfrak{t}(e_{2*}f_{1}^{*}\d) = [ X \overset{x}{\longrightarrow} Y \overset{y}{\longrightarrow} Z]$. Consider the following diagram of $\mathbb{G}$-triangles. 
\begin{equation}\label{2by2}
\begin{tikzcd}
X \arrow[r,"x"] \arrow[d,"e_{1}"]
& Y \arrow[r, "y"] \arrow[d, dashed, "r_{1}"]
& Z \arrow[r,dashed, "e_{2*}f_{1}^{*}\d"] \arrow[d,"f_{1}"]
& \text{}
\\
A \arrow[r,"a"]  \arrow[d,"e_{2}"] 
& B \arrow[r,"b"] \arrow[d,dashed, "r_{2}"] 
& C \arrow[r,dashed, "\d"] \arrow[d,"f_{2}"]
& \text{}
\\
X \arrow[r,"x"] 
& Y \arrow[r, "y"] 
& Z \arrow[r,dashed, "e_{2*}f_{1}^{*}\d"] 
& \text{}
\end{tikzcd}
\end{equation}
Observe that $$e_{1*}(e_{2*}f_{1}^{*}\d) = e_{*}f_{1}^{*}\d = f_{1}^{*}(e_{*}\d) = f_{1}^{*}(f^{*}\d) = f_{1}^{*}(f_{1}f_{2})^{*}\d = ((f_{2}f_{1})f_{1})^{*}\d = f_{1}^{*}\d.$$ Therefore $(e_{1},f_{1}) \colon e_{2*}f_{1}^{*}\d \rightarrow \d$ is a morphism of $\mathbb{G}$-extensions. 
So, by the axiom (ET2), there exists a morphism $r_{1} \colon Y \rightarrow B$ such that the top row of the above diagram commutes. 
Also observe that $$f_{2}^{*}(e_{2*}f_{1}^{*}\d) = e_{2*}f_{2}^{*}f_{1}^{*}\d = e_{2*}(f^{*}\d) = e_{*2}(e_{*}\d)= ((e_{2}e_{1})e_{2})_{*}\d = e_{2*}\d.$$ Therefore $(e_{2}, f_{2}) \colon \d \rightarrow e_{2*}f_{1}^{*}\d$ is a morphism of $\mathbb{G}$-extensions. So by axiom (ET2), there exists a morphism $r_{2} \colon B \rightarrow Y$ such that the bottom row of the above diagram commutes. Collapsing the above diagram into the diagram below, we obtain the following morphism of $\mathbb{G}$-triangles and commutative diagram.

\begin{equation}\label{automorph}
\begin{tikzcd}[column sep=3.0em,row sep=3.0em]
X \arrow[r,"x"] \arrow[d,equal, "e_{2}e_{1} = 1_{X}"]
& Y \arrow[r, "y"] \arrow[d,"r_{2}r_{1}"]
& Z \arrow[r,dashed, "e_{2*}f_{1}^{*}\d"] \arrow[d,equal,"f_{2}f_{1} = 1_{Z}"]
& \text{}
\\
X \arrow[r,"x"] 
& Y \arrow[r, "y"] 
& Z \arrow[r,dashed, "e_{2*}f_{1}^{*}\d"] 
& \text{}
\end{tikzcd}
\end{equation}
By Lemma \ref{2out3iso},  the morphism $r_{2}r_{1}$ is an automorphism of $Y$. That is to say, there exists $h \colon Y \rightarrow Y$ such that $r_{2}r_{1}h = 1_{Y}$ and $hr_{2}r_{1} = 1_{Y}$. Set $g := (r_{1}h)r_{2} \colon B \rightarrow B$. 
Observe that $$g^{2} = r_{1}h(r_{2}r_{1}h)r_{2} =  r_{1}h(1_{Y})r_{2} = r_{1}hr_{2} = g,$$
so $g$ is an idempotent morphism. Moreover, 
$$r_{2}(r_{1}h) = 1_{Y}.$$
So $g$ is in fact a split idempotent. Now consider the following diagram.
\begin{equation}\label{idemSplitFillingdiag1}
\begin{tikzcd}
A \arrow[r,"a"] \arrow[d,"e"]
& B \arrow[r,"b"] \arrow[d, dashed, "g"]
& C \arrow[d,"f"]
\\
A \arrow[r,"a"]
& B \arrow [r,"b"]
& C
\end{tikzcd}
\end{equation}
Note that, by the commutativity of diagram (\ref{automorph}), $$r_{2}r_{1}x = x$$ so $$x= (hr_{2}r_{1})x= hx,$$
and similarly  
$$yr_{2}r_{1} = y,$$ so $$y = y(r_{2}r_{1}h) = yh.$$ Using the fact that diagram (\ref{2by2}) commutes, we further observe that  
$$ga =r_{1}h(r_{2}a) = r_{1}(hx)e_{2} = (r_{1}x)e_{2} = a(e_{1}e_{2}) = ae,$$
and 
$$bg = (br_{1})hr_{2} = f_{1}(yh)r_{2} = (f_{1}y)r_{2} = (f_{1}f_{2})b = fb,$$
so diagram (\ref{idemSplitFillingdiag1}) commutes. This completes the proof. 
\end{proof}
\end{lemma}

\begin{definition}\cite[Definition 2.17]{NakaokaPalu}. Let $(\mathscr{M}, \mathbb{G}, \mathfrak{t})$ be an extriangulated category and $\mathscr{N}$ be a full additive subcategory closed under isomorphisms. The subcategory $\mathscr{N}$ is said to be \textit{extension-closed} if, for any conflation $A \rightarrow B \rightarrow C$ with $A,C \in \mathscr{N}$, we have that $B \in \mathscr{N}.$
\end{definition}

\begin{lemma}\label{extClosedSub}\cite[Remark 2.18]{NakaokaPalu}. Let $(\mathscr{M}, \mathbb{G}, \mathfrak{t})$ be an extriangulated category and let $\mathscr{N}$
be an extension-closed subcategory of $(\mathscr{M}, \mathbb{G}, \mathfrak{t})$. Let $\mathbb{G}_{\mathscr{N}}$ and $\mathfrak{t}_{\mathscr{N}}$ be the restrictions $\mathbb{G}$ and $\mathfrak{t}$ to $\mathscr{N}^{\text{op}} \times \mathscr{N}$. Then $(\mathscr{N}, \mathbb{G}_{N}, \mathfrak{t}_{\mathscr{N}})$ is an extriangulated category. 
\end{lemma}

Let $(\mathscr{C},\mathbb{E},\mathfrak{s})$ be an extriangulated category such that $\mathscr{C}$ is small. We have shown that the idempotent completion $\tilde{\mathscr{C}}$ is also an extriangulated category. Now consider the weak idempotent completion $\hat{\mathscr{C}}$. From the definition of the idempotent completion and the definition of the weak idempotent completion, it is easy to see that the weak idempotent completion $\hat{\mathscr{C}}$ (Definition \ref{weaklyIdemComp})  is a full subcategory of the idempotent completion $\tilde{\mathscr{C}}$ (Definition \ref{karoubian}). As we will see in the following proposition, the weak idempotent completion $\hat{\mathscr{C}}$ is extension-closed as a subcategory of the idempotent completion $\tilde{\mathscr{C}}$.\\

\begin{proposition}\label{weakExtClosed} Let $(\mathscr{C},\mathbb{E},\mathfrak{s})$ be an extriangulated category such that $\mathscr{C}$ is small and $(\tilde{\mathscr{C}}, \mathbb{F}, \mathfrak{r})$ be its idempotent completion. Then the weak idempotent completion $\hat{\mathscr{C}}$ of $\mathscr{C}$ is an extension-closed subcategory of $\tilde{\mathscr{C}}$. 

\begin{proof}
By Definitions \ref{karoubian}, \ref{weaklyIdemComp}, it is easy to see that $\hat{\mathscr{C}}$ is a full additive subcategory of $\tilde{\mathscr{C}}$. Now consider two  objects $(A,q) \in \hat{\mathscr{C}}$ and $(B,p) \in \tilde{\mathscr{C}}$. Suppose there exists an isomorphism $f \colon (A,q) \rightarrow (B,p) \in \tilde{\mathscr{C}}$. We will show that $(B,p)$ is also in $\hat{\mathscr{C}}$. Since $(A,q) \in \hat{\mathscr{C}}$, we have that $q$ is a split idempotent, which is to say there exists an object $X \in \mathscr{C}$ such that there are morphisms $q_{2} \colon A \rightarrow X,  q_{1} \colon X \rightarrow A$ satisfying $q =q_{1}q_{2}$ and $q_{2}q_{1} = 1_{X}$. 

Since  $f \colon (A,q) \rightarrow (B,p) \in \tilde{\mathscr{C}}$ is an isomorphism, there exists a morphism $f^{-1} \colon (B,p) \rightarrow (A,q)$ such that $ff^{-1} = p$ and $f^{-1}f=q$. Moreover $f$ satisfies the relations $fq = pf=f$ since it is a morphism in $\tilde{\mathscr{C}}$. Set $p_{1} = fq_{1} \colon X \rightarrow B$ and $p_{2} = q_{2}f^{-1}  \colon B \rightarrow X$, then we observe that $p = p_{1}p_{2}$ and $p_{2}p_{1} = 1_{X}.$
Hence, the idempotent $p$ splits, so $(B,p) \in \hat{\mathscr{C}}$. We conclude that $\hat{\mathscr{C}}$ is closed under isomorphisms. 

Let $(A,q)$ and $(C,p)$ be objects in $\hat{\mathscr{C}}$ and consider a conflation $(A,q) \overset{u^{\prime}}{\longrightarrow} (M,m) \overset{v^{\prime}}{\longrightarrow} (C,p)$  of some $\mathbb{F}$-extension $\d = p^{*}q_{*}\varepsilon$ in $\mathbb{F}((C,p),(A,q))$. We will show that $(M,m)$ is also in $\hat{\mathscr{C}}$. By Lemma \ref{standForm}, $u^{\prime} = uq$ and $v^{\prime} = pv$ for some $u \colon A \rightarrow M$ and $v \colon M \rightarrow C$, so
$$ \mathfrak{r}(\d) = [(A,q) \overset{uq}{\longrightarrow} (M,m) \overset{pv}{\longrightarrow} (C,p)].$$
Suppose that $$\mathfrak{s}(\d)=[ A \overset{x}{\longrightarrow} B \overset{y}{\longrightarrow} C]$$ then $$\mathfrak{r}(\d) = [(A,q) \overset{xq}{\longrightarrow} (B,r) \overset{py}{\longrightarrow} (C,p)].$$

By definition we have the following equivalence, 
\begin{center}
\begin{tikzcd}
(A,q) \arrow[r,"uq"] \arrow[d,equal]
& (M,m) \arrow[r,"pv"] \arrow[d,"h"]
& (C,p) \arrow[d,equal]
\\
(A,q) \arrow[r,"xq"]
& (B,r) \arrow[r,"py"]
& (C,p)
\end{tikzcd}
\end{center}
for some isomorphism $h \colon (M,m) \rightarrow (B,r)$. Since $(A,q), (C,p) \in \hat{\mathscr{C}}$, we have that the idempotents $q$ and $p$ are split idempotents. Since $ q_{*}\d = p^{*}\d$ we have that $(q,p) \colon \d \rightarrow \d$ is a morphism of $\mathbb{F}$-extensions. Therefore by Lemma \ref{SplitIdemFillingMorph}, there exists a split idempotent $w \colon B \rightarrow B$ such that $wx = xq$ and $py = yw$. By Proposition \ref{independentOfRep}, we have the following equivalence 
\begin{center}
\begin{tikzcd}
(A,q) \arrow[r,"xq"] \arrow[d,equal]
& (B,r) \arrow[r,"py"] \arrow[d,"g"]
& (C,p) \arrow[d,equal]
\\
(A,q) \arrow[r,"xq"]
& (B,w) \arrow[r,"py"]
& (C,p)
\end{tikzcd}
\end{center}
for some isomorphism $g \colon (B,r) \rightarrow (B,w)$. The composition $gh \colon (M,m) \rightarrow (B,w)$ is an isomorphism. Moreover $(B,w)$ lies in $\hat{\mathscr{C}}$ since $w$ is a split idempotent, so it follows that $(M,m) \in \hat{\mathscr{C}}$ since $\hat{\mathscr{C}}$ is closed under isomorphims. This completes the proof. 
\end{proof}
\end{proposition}

\begin{theorem}
Let $(\mathscr{C},\mathbb{E},\mathfrak{s})$ be an extriangulated category such that $\mathscr{C}$ is small. Let $\hat{\mathscr{C}}$ be the weak idempotent completion of $\mathscr{C}$. Then $\hat{\mathscr{C}}$ is extriangulated. Moreover, the embedding $j_{\mathscr{C}} \colon \mathscr{C} \rightarrow \hat{\mathscr{C}}$ is an extriangulated functor.
\begin{proof}
Since the idempotent completion $\tilde{\mathscr{C}}$ is extriangulated and the weak idempotent completion $\hat{\mathscr{C}}$ is an extension-closed subcategory of $\tilde{\mathscr{C}}$ by Proposition \ref{weakExtClosed}, $\hat{\mathscr{C}}$ is also extriangulated by Lemma \ref{extClosedSub}. The other statement follows as an easy consequence of this and Theorem \ref{IdempCompExtri}.
\end{proof}
\end{theorem}

\bibliographystyle{alpha}
	\bibliography{IdempotentCompletion}
\end{document}